\documentclass[AMA,STIX1COL]{WileyNJD-v2}
\usepackage{moreverb}
\newcommand\BibTeX{{\rmfamily B\kern-.05em \textsc{i\kern-.025em b}\kern-.08em
T\kern-.1667em\lower.7ex\hbox{E}\kern-.125emX}}
\articletype{Research Article}
\received{<day> <Month>, <year>}
\revised{<day> <Month>, <year>}
\accepted{<day> <Month>, <year>}
\newcommand{\mc}[1]{\mathcal {#1}}
\newcommand{\dg}{{\dagger}}

\newcommand{\dt}{{\tiny{\odot}} }

\newcommand{\ba}{{\bf{a}}}
\newcommand{\bb}{{\bf{b}}}
\newcommand{\bc}{{\bf{c}}}
\newcommand{\bx}{{\bf{x}}}
\newcommand{\by}{{\bf{y}}}
\newcommand{\bz}{{\bf{z}}}
\newcommand{\bu}{{\bf{u}}}
\newcommand{\bv}{{\bf{v}}}

 \def\rn{{\mathbb R}^{p}}
\def\rt{{\mathbb{R}^{n\times n \times n_3\times \cdots\times n_p}}}

\usepackage{hyperref}
\usepackage{comment}
\newtheorem{example}{Example}
\usepackage{subfigure}

\usepackage{tikz,xcolor}
\definecolor{lime}{HTML}{A6CE39}
\definecolor{lightblue}{rgb}{0.0, 0.0, 0.5}
\DeclareRobustCommand{\orcidicon}{%
	\begin{tikzpicture}
	\draw[lime, fill=lime] (0,0)
	circle [radius=0.16]
	node[white] {{\fontfamily{qag}\selectfont \tiny ID}};
	\draw[white, fill=white] (-0.0625,0.095)
	circle [radius=0.007];
	\end{tikzpicture}
	\hspace{-2mm}
}
\foreach \x in {A, ..., Z}{%
	\expandafter\xdef\csname orcid\x\endcsname{\noexpand\href{https://orcid.org/\csname orcidauthor\x\endcsname}{\noexpand\orcidicon}}
}

\begin{document}

\title{
Computation of Generalized Inverses of Tensors via $t$-Product
}

\author[1]{Ratikanta Behera}

\author[2]{Jajati Keshari Sahoo}

\author[1]{R. N. Mohapatra}

\author[1]{M. Zuhair Nashed}

\authormark{RATIKANTA BEHERA \textsc{et al}}

\address[1]{\orgdiv{Department of Mathematics}, \orgname{University of Central Florida, Orlando}, \orgaddress{\state{Florida}, \country{USA}}}

\address[2]{\orgdiv{Department of Mathematics}, \orgname{BITS Pilani, K.K. Birla Goa Campus}, \orgaddress{\state{Goa}, \country{India}}}

\corres{R.N. Mohaptra, Department of Mathematics, University of Central Florida, USA. \email{ram.mohapatra@ucf.edu }\\ 
Jajati Keshari Sahoo, Department of Mathematics, BITS Pilani, K.K. Birla Goa Campus, Goa, India. \email{jksahoo@goa.bits-pilani.ac.in}\\
\\
{\bf Funding Information}\\
 Mohapatra Family Foundation and the College of Graduate Studies, University of Central Florida, Orlando, 32826, Florida, USA}

\abstract[Abstract]{
Generalized inverses of tensors play increasingly important
roles in computational mathematics and numerical analysis. It is appropriate to develop the theory of generalized inverses of tensors within the algebraic structure of a ring. In this paper, we study different generalized inverses of tensors over a commutative ring and a non-commutative ring.  Several numerical examples are provided in support of the theoretical results. We also propose algorithms for computing the inner inverses, the Moore-Penrose inverse, and weighted Moore-Penrose inverse of tensors over a non-commutative ring. The prowess of some of the results is demonstrated by applying these ideas to solve an image deblurring problem.
}

\keywords{ Ring, Generalized inverses, Tensor, T-product, Moore-Penrose inverse, Image deblurring.
\vskip 0.1cm
{\bf AMS SUBJECT CLASSIFICATIONS:}
15A09,~1569}

\maketitle

\section{Introduction}
One of the basic operations in linear algebra is matrix multiplication;  it expresses the product of two matrices to form a new matrix. A tensor is a higher-dimensional generalization of a matrix (i.e., a first-order tensor is a vector, a second-order tensor is a matrix), but there are more than one way to multiply two tensors \cite{Braman10, BaderTam06}. Even some basic matrix multiplication concepts cannot be generalized to tensor multiplication in a unique manner \cite{Kolda09Rev, Kruskal77}. This leads to the study of different types of tensor products \cite{BaderTam06,MiaYimin20},  which have recently attracted a great deal of interest (see \cite{MiaYimin20, Zhu15, Kolda09Rev, miao2020t, wang2020robust, Ragn12}). In this connection, the inverses and generalized inverses over different product of tensors  \cite{jin2017,SahooBeS19,BehNS19, Behera2018} have generated a tremendous amount of interest in mathematics, physics, computer science, and engineering.

In mathematics, a ring is an algebraic structure with two binary operations (addition and multiplication) over a set satisfying certain requirements. This structure facilitates the fundamental physical laws, such as those underlying special relativity and symmetry phenomena in molecular chemistry. In the literature, many different approaches have been proposed for the generalized inverse of a ring \cite{rakic, Koliha07, Zhu15}, where elements of the ring are scalars, vectors, or matrices. However, it is impractical to analyze quantities whose elements are addressed by more than two indices in a ring, (for example, a vector has one index and a matrix has two indices). This inconvenience can be easily overcome, thanks to tensors \cite{BaderTam06, Kolda09Rev}.  In connection with binary operations of the algebraic structure, we consider $+$ for addition and use two different binary operations, i.e.,  $*$ and $\dt$, for multiplication. In this paper, we study the generalized inverses of tensors over a commutative ring with the binary operations $(+, \dt)$ and a non-commutative ring with involution using $(+, *)$ as a binary operations.

Kilmer and Martin \cite{kilmer11} proposed a closed multiplication operation between tensors referred to as the t-product. In fact, the multiplication of two tensors based on the $t$-product takes advantage of the circulant-type structure \cite{kilmer11, Kilmer13}, which allows one to compute efficiently using the Fast Fourier Transform and extend many concepts from linear algebra to tensors. 
 
A summary of the main facets of this paper are given below by the following bullet points.
\begin{enumerate}
 \setlength{\parskip}{0pt}
\item[$\bullet$] Introduction of different generalized inverses of tensors (via $t$-product) over a commutative ring and a non-commutative ring. 
\item[$\bullet$] Determination of necessary and sufficient conditions for reverse order law for the Moore-Penrose inverse and $\{i, j, k\}$ -inverses of tensors over a non-commutative ring.
\item[$\bullet$]  Discussion of a few algorithms for computing generalized inverse, the Moore-Penrose inverse and weighted Moore-Penrose inverse of tensors over a non-commutative ring.
\item[$\bullet$] Application of generalized inverses to colour image deblurring.
\end{enumerate}

Let $V$ and $W$ be vector spaces (both real or both complex). Let $A$ be a linear transformation from $V$ into $W$. Denote the null space of $A$ and the range of $A$ by $\mathcal{N}(A)$ and $\mathcal{R}(A),$ respectively. Let $M$ be an algebraic complement of $\mathcal{N}(A)$ in $V$ and $S$ be an algebraic complement of $\mathcal{R}(A)$ in $W$, so $V = \mathcal{N}(A) \oplus  M$ and $W = \mathcal{R}(A) \oplus S.$ Let $P$ be the algebraic projector (linear idempotent transformation) of $V$ onto $M$, and let $Q$ be the algebraic projector of $W$ onto $S$.  Then there exists a unique linear transformation $X$ on $W$ into $V$ that satisfies the following equations:
\begin{equation*}
AXA = A, ~~~~XAX = X, ~~~~XA = P~~ \mbox{and}~~ AX = Q. 
\end{equation*}
The transformation $X$ is called the (algebraic) generalized inverse of $A$ relative to the projectors $P$ and $Q,$ and is denoted by $A^{\sharp}_{P,Q}$. In the case when $A$ is a bounded linear operator on a Hilbert space $V$ into Hilbert space $W$, the choice of orthogonal projector $P$ and $Q$ gives the Moore-Penrose inverse. {\it ``Inverses''} that satisfy a subset of the above four equations are of interest in many applications. A linear operator $X$ is called an {\it inner inverse} if it satisfies $AXA =A$ and an {\it outer inverse} if it satisfies $XAX=X$.  Other {\it ``inverses"} are defined to satisfy one of the four equations coupled with additional equations, such as the Drazin inverse. Various generalized inverses serve different purposes. For detailed expositions of  generalized inverses of matrices including computational aspects and applications, we refer to the following  books: Ben-Israel and Greville \cite{Ben-Israel}, Rao and Mitra \cite{raobook},  Campbell and Meyer \cite{Campbell91}.  Wei et al. \cite{Weibook}, and Nashed \cite{nashed2014generalized, Nashed87,Nashed70}.
A comprehensive annotated bibliography of 1876 references is included in Nashed \cite{NASHED19761}. For the theory of generalized inverses on (not necessarily finite dimensional) vector spaces and on Banach and Hilbert spaces, see Nashed \cite{Nashed87}. For various equivalent definitions of the Moore-Penrose of matrices and linear operators, see Nashed \cite{Nashed71}.

The purpose of this paper is twofold. Firstly, we study generalized inverses of tensors over a commutative ring. This part is motivated by the work of Drazin \cite{Drazin12, Drazin16, Darzin20} and Zhu \cite{Ke18, Huihu18} and Raki\'c \cite{rakic} in rings with involution. Here we relate the concept of a well-supported element in a ring with involution and study the group inverse of tensors along with the class of $(\textbf{b}, \textbf{c})$-inverses in a commutative ring. Secondly, we discuss the generalized inverse of tensors over non-commutative ring, which is an extension of a matrix over a ring (see \cite{Bapat90, Manju94, Zhu15}, also see the recent papers on tensor \cite{Braman10,BramanK07,liangBing2019, jin2017}). Specifically, we focus on investigations of $\{i, j, k\}$-inverses, the Moore-Penrose inverse and the weighted Moore-Penrose inverse of tensors along with several characterizations of these inverses in a non-commutative ring with involution. Our aim is to focus on generalized inverses of tensors with a specific multiplication concept over a  ring. We present algorithms for computing different generalized inverses of tensors in a non-commutative ring.  This algorithms may open to the door to other types of tensor related problems. As an application, we use the tensor representation algorithm in image deblurring.

The outline of the paper is as follows. Section 2 introduces different generalized inverses of tensors over a commutative ring. In Section 3, we introduce different generalized inverses of tensors over a non-commutative ring with involution. Further, we discuss a few necessary and sufficient conditions for existence of such inverses in this section. We also develop algorithms for computing generalized inverses over a non-commutative ring. An application of the Moore-Penrose inverse on image reconstruction is discussed in Section 4. Section 5 devoted to a brief conclusion. 

\section{Tensors over commutative ring}
It is well-known that  the tensors are a multi-dimensional array of numbers, sometimes it is called a multi-way or a multi-mode array.  For example, a  matrix is a second order tensor or a  two way tensor. 
An element $\textbf{a}=(a_1, a_2, \cdots,  a_p)^T,$ where the entries $a_1, a_2, \cdots,a_p$ are elements from the field of real number $\mathbb{R}$, is called a first order tensor with entries from $\mathbb{R}$. Following standard notation, the set of all first order tensors with entries from the field of real number $\mathbb{R}$ is denoted by $\mathbb{R}^p$. This section is focused on $\mathbb{R}^p$. The circulant of an element $\textbf{a}=(a_1, a_2, \cdots a_{p})^T \in \mathbb{R}^p$ is defined as 
\begin{equation*}
 \text{circ}(\bf{a}) = 
\begin{bmatrix}
a_1 & a_{p} & \cdots & a_2\\
a_2 & a_1    &  \cdots & a_3\\
\vdots  &  \vdots  & \vdots  &    \vdots\\
a_{p}& a_{p-1}& \cdots & a_1
 \end{bmatrix}.
\end{equation*}
We recall the following result which will be used in Section 2.1.
\begin{lemma}{(\cite{BramanK07}, Theorem 3.2)} $(\mathbb{R}^p , +,  \dt)$ is a commutative ring with unity, where the addition `$+$', multiplication `$\odot$' and the unity respectively defined by 
\begin{eqnarray*}
\ba+\bb &=&  (a_1, a_2, \cdots a_p)^T +(b_1, b_2, \cdots, b_p)^T=(a_1+b_1, a_2+b_2, \cdots, a_p+b_p)^T,\\
\ba\odot\bb &=&  \text{circ}({\bf{a}})\bb = 
\begin{bmatrix}
a_1 & a_{p} & \cdots & a_2\\
a_2 & a_1    &  \cdots & a_3\\
\vdots  &  \vdots  & \vdots  &    \vdots\\
a_{p}& a_{p-1}& \cdots & a_1
 \end{bmatrix}
 \begin{bmatrix}
b_1\\
b_2\\
\vdots\\
b_{p}
\end{bmatrix},~~\textnormal{~and~} {\bf{i}}=(1,0,0 \cdots,0)^T \textnormal{~is the unity of~} \mathbb{R}^p. 
\end{eqnarray*}

\end{lemma}
\begin{definition}
Let  $(\mathbb{R}^p , +,  \dt)$  be a commutative ring. An element  $\ba (\neq {\bf{0}}) \in \mathbb{R}^p$ is said to be a zero-divisor if there exists an element  $\bb (\neq {\bf{0}}) \in \mathbb{R}^p$, such that $\ba \odot \bb  ={\bf{0}}$.
\end{definition}

\begin{remark}
 $(\mathbb{R}^p , +,  \dt)$ is not an integral domain as shown in the next example.
\end{remark}

\begin{example}
Let $\ba = (a, a, a)^T$ with $a\neq 0$. Consider $\bb =(1, -1, 0)^T\in \mathbb{R}^3$. Then 
\begin{equation*}
\ba \odot \bb =  \text{circ}(\ba)\bb=
\begin{bmatrix}
a & a & a\\
a & a  &  a\\
a & a & a
 \end{bmatrix} 
\begin{bmatrix}
1 \\
-1 \\
0
 \end{bmatrix}=
 \begin{bmatrix}
0 \\
0 \\
0
 \end{bmatrix}
 ={\bf{0}}.
\end{equation*}
\end{example}
Thus $\ba$ is zero-divisor in $\mathbb{R}^p$.  

\subsection{Generalized Inverses}
In this section we discuss different types of generalized inverses on $(\mathbb{R}^p , +,  \dt)$. For convenience we use the known notation $\mathbb{R}^p$ for the ring $(\mathbb{R}^p , +,  \dt)$. 
\begin{definition}
Let $\textbf{a} \in \mathbb{R}^p$. The left and right ideals generated by $\textbf{a}$ are denoted by $\mathbb{R}^p\dt\textbf{a}$ and $\textbf{a}\dt\mathbb{R}^p$  which are respectively defined as
\begin{equation*}
\mathbb{R}^p\dt \textbf{a}=\{ \textbf{x}\dt\textbf{a} ~:\textbf{x}\in \mathbb{R}^p \} ~\textnormal{and}~
    \textbf{a}\dt \mathbb{R}^p=\{ \textbf{a}\dt \textbf{x} ~:\textbf{x}\in \mathbb{R}^p \}.
\end{equation*}
\end{definition}
We also denote  $\ba\dt\mathbb{R}^p\dt\bb =\{\ba\dt\bx\dt\bb:~\bx \in\mathbb{R}^p\}$. The left and right annihilator of an element  are defined as follows.
\begin{definition}
The left and right annihilator of $\ba\in\rn$, respectively denoted by $\textup{lann}(\textbf{a})$ and
 $\textup{rann}(\textbf{a})$, are defined as 
\begin{equation*}
\textup{lann}(\textbf{a})=\{\textbf{y}\in \mathbb{R}^p~: \textbf{y}\dt \textbf{a}=\textbf{0} \}\mbox{ and }\textup{rann}(\textbf{a})=\{ \textbf{y}\in \mathbb{R}^p~:\textbf{a}\dt \textbf{y}=\textbf{0} \}.
\end{equation*}
\end{definition}

\begin{definition}\label{bcinverse}
Let $\ba, \bb, \bc\in\mathbb{R}^p$. An element $\by\in\mathbb{R}^p$ is called a $(\bb, \bc)$-inverse of $\ba$ if it satisfies 
\begin{enumerate}
 \setlength{\parskip}{0pt}
    \item[(a)] $\by\in\bb\dt\mathbb{R}^p\dt\by\cap\by\dt\mathbb{R}^p\dt\bc$,
    \item[(b)] $\by\dt\ba\dt\bb=\bb$ and $\bc\dt\ba\dt\by=\bc$.
\end{enumerate}
\end{definition}

\begin{example}\rm
Let $\ba=(0,1)^t=\bc, \bb=(1,0)^t\in\mathbb{R}^2$. Consider an element $\by=(0,1)^t\in\mathbb{R}^2$. We can verify that 
\begin{center}
    $\ba\dt\bb= \text{circ}(\ba)\bb =\begin{bmatrix}
       0 & 1\\
       1 & 0
    \end{bmatrix} \begin{bmatrix}
         1\\
         0
    \end{bmatrix}=\begin{bmatrix}
         0\\
         1
    \end{bmatrix} \mbox{ ~and~ } \bc\dt\ba= \text{circ}(\bc)\bb =\begin{bmatrix}
       0 & 1\\
       1 & 0
    \end{bmatrix} \begin{bmatrix}
         0\\
         1
    \end{bmatrix}=\begin{bmatrix}
         1\\
         0
    \end{bmatrix}$,
\end{center}

\begin{center}
    $\by\dt(\ba\dt\bb)=\begin{bmatrix}
       0 & 1\\
       1 & 0
    \end{bmatrix} \begin{bmatrix}
         0\\
         1
    \end{bmatrix}=\begin{bmatrix}
         1\\
         0
    \end{bmatrix}=\bb \mbox{ ~and~ } (\bc\dt\ba)\dt\by=\begin{bmatrix}
       1 & 0\\
       0 & 1
    \end{bmatrix} \begin{bmatrix}
         0\\
         1
    \end{bmatrix}=\begin{bmatrix}
         0\\
         1
    \end{bmatrix}=\bc$.
\end{center}
Further, $\bb\dt\bf{u}\dt\by=\by$ and $\by\dt\bf{v}\dt\bc=\by$ for an ${\bf{u}}=(1,0)^t$ and a ${\bf{v}}=(0,1)^t$. Thus $\by$ is the $(\bb, \bc)$-inverse of $\ba$. 
\end{example}
If $\by$ is the $(\bb,\bc)$ inverse of $\ba\in\rn$, then we have
\begin{equation*}
    \bb=\by\dt\ba\dt\bb=\bb\dt\bf{s}\dt\by\dt\ba\dt\bb=\bb\dt\bx\dt\bb, \mbox{~where} ~\bf{s}\in\rn ~~\mbox{and} ~~\bx=\bf{s}\dt\by\dt\ba\in\rn.
\end{equation*}
Similarly, $\bc=\bc\dt\bf{z}\dt\bc$ for some $\bz\in\rn$. This leads to the following result.

\begin{proposition}
Let $\ba,~\bb,~\bc\in\rn$. If the $(\bb, \bc)$ inverse of an element $\ba$ exists, then $\bb$ and $\bc$ are both regular.
\end{proposition}

Next, we discuss the existence  of $(\textbf{b}, \textbf{c})$-inverse.

\begin{theorem}\label{thm21.2}
Let $\ba, \bb, \bc\in\mathbb{R}^p$. Then  the $(\bb, \bc)$-inverse of $\ba$ exists if and only if
$\bb,\bc\in\mathbb{R}^p\dt\bc\dt\ba\dt\bb=\bc\dt\ba\dt\bb\dt\mathbb{R}^p.$
\end{theorem}
\begin{proof}
Let $\bx$ be the $(\bb, \bc)$-inverse of $\ba$. Then $\bb=\bx\dt\ba\dt\bb$, $\bc\dt\ba\dt\bx=\bc$ and $\bx=\bx\dt\bf{u}\dt\bc=\bb\dt\bf{v}\dt\bx$ for some $\bf{u}$ and $\bf{v}\in\mathbb{R}^p$. Now $\bb=\bx\dt\bf{u}\dt\bc\dt\ba\dt\bb\in\mathbb{R}^p\dt\bc\dt\ba\dt\bb$. Similarly, $\bc=\bc\dt\ba\dt\bx=\bc\dt\ba\dt\bb\dt\bf{v}\dt\bx\in \bc\dt\ba\dt\bb\dt\mathbb{R}^p$. Conversely, let $\bb\in\mathbb{R}^p\dt\bc\dt\ba\dt\bb$ and $\bc\in \bc\dt\ba\dt\bb\dt\mathbb{R}^p$. Then $\bb=\bf{u}\dt\bc\dt\ba\dt\bb$ and $\bc=\bc\dt\ba\dt\bb\dt\bf{v}$ for some $\bf{u},\bf{v}\in\mathbb{R}^p$. Further, $\bf{u}\dt\bc=\bb\dt\bf{v}$. Let $\bx=\bf{u}\dt\bc=\bb\dt\bf{v}$. Now 
\begin{center}
$\bb=\bf{u}\dt\bc\dt\ba\dt\bb=\bf{u}\dt(\bc\dt\ba\dt\bb\dt\bf{v})\dt\ba\dt\bb=\bb\dt\bf{v}\dt\ba\dt\bb=\bx\dt\ba\dt\bb$, \\
    $\bc=\bc\dt\ba\dt\bb\dt\bf{v}=\bc\dt\ba\dt(\bf{u}\dt\bc\dt\ba\dt\bb)\dt\bf{v}=\bc\dt\ba\dt\bf{u}\dt\bc=\bc\dt\ba\dt\bx$, \\
    $\bx=\bb\dt\bf{v}=\bb\dt\bf{v}\dt\ba\dt\bb\dt\bf{v}=\bb\dt\bf{v}\dt\ba\dt\bx\in\bb\dt\mathbb{R}^p\dt\bx$, and\\
    $\bx=\bf{u}\dt\bc=\bf{u}\dt\bc\dt\ba\dt\bf{u}\dt\bc=\bx\dt\ba\dt\bf{u}\dt\bc\in\bx\dt\mathbb{R}^p\dt\bc$.
\end{center}
Thus $\bx$ is a $(\bb, \bc)$-inverse of $\ba$.
\end{proof}
We now establish the uniqueness of the $(\bb,\bc)$-inverse. \begin{theorem}\label{unibc}
Let $\ba, \bb, \bc\in\mathbb{R}^p$. 
Then the $(\bb, \bc)$-inverse of $\ba$ is  unique. 
\end{theorem}
\begin{proof}
Suppose there exist two $(\bb, \bc)$-inverses of $\ba$, say $\bx$ and $\bf{y}$. From Definition \ref{bcinverse}, we obtain $\bx\dt\ba\dt\bb=\bb,~\bc\dt\ba\dt\bf{y}=\bc,~\bx=\bx\dt\bf{u}\dt\bc,~\bf{y}=\bb\dt\bf{v}\dt\bf{y},~\bx=\bb\dt\bf{v}\dt\bx$ for some $\bf{u}$ and $\bf{v}\in\mathbb{R}^p$. Now 
\begin{center}
    $ \bx=\bx\dt\bf{u}\dt\bc=\bx\dt\bf{u}\dt\bc\dt\ba\dt\bf{y}=\bx\dt\ba\dt\bf{y}=\bx\dt\ba\dt\bb\dt\bf{v}\dt\bf{y}=\bb\dt\bf{v}\dt\bf{y}=\bf{y}$.
\end{center}
\end{proof}
Note that $\bx=\bb\dt\bf{v}\dt\bx=\bx\dt\ba\dt\bb\dt\bf{v}\dt\bx=\bx\dt\ba\dt\bx$.
\begin{theorem}
Let $\ba, \bb, \bc\in\mathbb{R}^p$. Then  the following statements are equivalent:
\begin{enumerate}
 \setlength{\parskip}{0pt}
    \item[(i)] $(\bb, \bc)$-inverse of $\ba$ exists.
    \item[(ii)] $\mathbb{R}^p=\ba\dt\bb\dt\mathbb{R}^p\bigoplus \textup{rann}(\bc)=\mathbb{R}^p\dt\bc\dt\ba\bigoplus\textup{lann}(\bb)$.
    \item[(iii)] $\mathbb{R}^p=\ba\dt\bb\dt\mathbb{R}^p+ \textup{rann}(\bc)=\mathbb{R}^p\dt\bc\dt\ba+\textup{lann}(\bb)$.
\end{enumerate}
\end{theorem}

\begin{proof}
It is enough to show $(i)\Rightarrow (ii)$ and $(iii)\Rightarrow (i)$ since $(ii)\Rightarrow (iii)$ is trivial. \\
$(i)\Rightarrow (ii)$ Let the  $(\bb,\bc)$ inverse of  $\ba$ exists. Then
 \begin{equation*}
      \bb=\bf{u}\dt\bc\dt\ba\dt\bb \mbox{ and  }\bc=\bc\dt\ba\dt\bb\dt\bf{v} \mbox{   for some  }  \bf{u},\bf{v}\in\rn.
 \end{equation*}

Let $\bx\in\rn$ be any arbitrary vector and consider $\bf{z}=\bx-\ba\dt\bb\dt\bf{v}\dt\bx$. Then $\bc\dt\bf{z}=(\bc-\bc\dt\ba\dt\bb\dt\bf{v})\dt\bx=\textbf{0}$.
Thus $\bz\in\textup{rann}(c)$. In addition,  $\bx=\ba\dt\bb\dt\bf{v}\dt\bx+\bz\in \ba\dt\bb\dt\rn+\textup{rann}(\bc)$. This implies 
\begin{equation*}
\mathbb{R}^p=\ba\dt\bb\dt\rn+\textup{rann}(\bc).
\end{equation*}
 Further, if $\by\in \ba\dt\bb\dt\rn\cap\textup{rann}(\bc)$, then $\by=\ba\dt\bb\dt\bf{t}$ and $\bc\dt\by=\textbf{0}$ for some $\bf{t}\in\rn$. However,  $\bb\dt\bf{t}=\bu\dt\bc\dt\ba\dt\bb\dt\bf{t}=\bu\dt\bc\dt(\ba\dt\bb\dt\bf{t})=\bu\dt\bc\dt\by=\textbf{0}$. Therefore, $\by=\textbf{0}$, which implies $\ba\dt\bb\dt\rn+\textup{rann}(\bc)=\{\textbf{0}\}$ and hence $\mathbb{R}^p=\ba\dt\bb\dt\rn\bigoplus\textup{rann}(\bc)$. Similarly, we can show the other equality $\mathbb{R}^p=\rn\dt\bc\dt\ba\bigoplus\textup{lann}(\bb)$. \\
 $(iii)\Rightarrow (i)$ Let $\rn=\ba\dt\bb\dt\rn+\textup{rann}(\bc)$.  Then $\bc=\bc\dt \bf{i}=\bc\dt(\ba\dt\bb\dt\bu+\bv)$, where $\bu\in\rn$ and $\bv\in\textup{rann}(\bc)$. From $\bv\in\textup{rann}(\bc)$, we get $\bc\dt\bv=\bf{0}$. This yields $\bc=\bc\dt\ba\dt\bb\dt\bu\in\bc\dt\ba\dt\bb\dt\rn$. Using $\rn\dt\bc\dt\ba+\textup{lann}(\bb)$, we can show $\bb\in\rn\dt\bc\dt\ba\dt\bb$. Thus by Theorem \ref{thm21.2}, the $(\bb,\bc)$ inverse of $\ba$ exists.
\end{proof}

We next define the annihilator inverses: 

\begin{definition}
Let $\ba, \bb, \bx\in\mathbb{R}^p$. An element $\bx$ is called a left annihilator  $\bb$-inverse of $\ba$ if it satisfies 
\begin{center}
    $\bx\dt\ba\dt\bx=\bx,~\textup{lann}(\bx)=\textup{lann}(\bb)$.
\end{center}
\end{definition}

\begin{definition}
Let $\ba, \bc, \bx\in\mathbb{R}^p$. An element $\bx$ is called a right annihilator  $\bc$-inverse of $\ba$ if it satisfies 
\begin{center}
    $\bx\dt\ba\dt\bx=\bx,~\textup{rann}(\bx)=\textup{rann}(\bc)$.
\end{center}
\end{definition}

The uniqueness of  the annihilator inverse is discussed in the following result.
\begin{theorem}
Let $\ba, \bb, \bc\in\mathbb{R}^p$. If a left annihilator $\bb$-inverse (or right annihilator  $\bc$-inverse) of $\ba$ is exists then it is unique. 
\end{theorem}
\begin{proof}
Assume to the contrary, let $\bx$ and $\by$ be left annihilator $\bb$-inverse  of $\ba$. Then
\begin{center}
  $\bx\dt\ba\dt\bx=\bx, \by\dt\ba\dt\by=\by,~\textup{lann}(\bx)=\textup{lann}(\bb)=\textup{lann}(\by)$.  
\end{center}
From $\bf{i}-\bx\dt\ba\in\textup{lann}(\bx)=\textup{lann}(\by)$, we get $\by=\bx\dt\ba\dt\by$. Similarly, we can show
$\bx=\by\dt\ba\dt\bx$. This further implies $\bx=\bx\dt\ba\dt\by$ due to the fact that $\rn$ is commutative. 
Hence $\bx=\by$. Further, if $\bx$ and $\by$ are two right annihilator $\bc$-inverse  of $\ba$. Then
\begin{center}
  $\bx\dt\ba\dt\bx=\bx, \by\dt\ba\dt\by=\by,~\textup{rann}(\bx)=\textup{rann}(\bc)=\textup{rann}(\by)$.  
\end{center}
From $\bf{i}-\ba\dt\bx\in\textup{rann}(\bx)=\textup{rann}(\by)$, we get $\by=\by\dt\ba\dt\bx$. Similarly, we can show
    $\bx=\bx\dt\ba\dt\by=\by\dt\ba\dt\bx$. Therefore, $\bx=\by$.
\end{proof}
We now present an equivalent characterization of left  annihilator $\bb$-inverse.

\begin{lemma}\label{lem2.18}
Let $\ba, \bb, \bx\in \rn$. Then the following  are equivalent:
\begin{enumerate}
 \setlength{\parskip}{0pt}
    \item[(i)] $\bx\dt\ba\dt\bx=\bx,~\textup{lann}(\bx)=\textup{lann}(\bb)$.
    \item[(ii)] $\bx\dt\ba\dt\bb=\bb,~\textup{lann}(\bb)\subseteq\textup{lann}(\bx)$.
\end{enumerate}
\end{lemma}
\begin{proof}
$(i)\Rightarrow (ii)$ It is enough to show only $\bx\dt\ba\dt\bb=\bb$. Let $\bx\dt\ba\dt\bx=\bx$. Then $\bf{i}-\bx\dt\ba\in \textup{lann}(\bx)=\textup{lann}(\bb)$. This leads $\bb=\bx\dt\ba\dt\bb$.\\
$(ii)\Rightarrow (i)$ Let $\bb=\bx\dt\ba\dt\bb$. Then $\bf{i}-\bx\dt\ba\in\textup{lann}(\bb)\subseteq \textup{lann}(\bx)$. Thus $\bx\dt\ba\dt\bx=\bx$. Next we will claim that $\textup{lann}(\bx)\subseteq\textup{lann}(\bb)$. Let $\bv\in\textup{lann}(\bx)$. Then $\bv\dt\bx=\bf{0}$ and $\bv\dt\bb=\bv\dt\bx\dt\ba\dt\bb=\bf{0}$. Therefore, $\bv\in\textup{lann}(\bb)$ and hence $\textup{lann}(\bx)\subseteq\textup{lann}(\bb)$.
\end{proof}
The next result for right annihilator $\bc$-inverse can be proved in a similar way.
\begin{lemma}\label{lem2.19}
Let $\ba, \bc, \bx\in \rn$. Then the following  are equivalent:
\begin{enumerate}
 \setlength{\parskip}{0pt}
    \item[(i)] $\bx\dt\ba\dt\bx=\bx,~\textup{rann}(\bx)=\textup{lann}(\bc)$.
    \item[(ii)] $\bc\dt\ba\dt\bx=\bc,~\textup{rann}(\bc)\subseteq\textup{rann}(\bx)$.
\end{enumerate}
\end{lemma}
Since $\rn$ is a associative ring, we define the annihilator $(\bb,\bc)$-inverse as follows. 
\begin{definition}
Let $\ba, \bb, \bc, \bx\in \rn$. An element $\bx\in\rn$ is called an annihilator $(\bb,\bc)$-inverse of $\ba$ if it satisfies 
\begin{center}
   $\bx\dt\ba\dt\bx=\bx,~\textup{lann}(\bx)=\textup{lann}(\bb),~\textup{rann}(\bx)=\textup{rann}(\bc)$.
\end{center}
\end{definition}
We now give an example of the annihilator $(\bb,\bc)$-inverse.
\begin{example}\rm
Let $\ba=\left(\frac{2}{3}, \frac{1}{3}, 0 \right)^T\in \mathbb{R}^3$ and $\bb=\bc=(1, -1, 0)^T\in\mathbb{R}^3$. Consider $\bx=(1, -1, 0)^T\in\mathbb{R}^3$. Then we can verify that  
\begin{center}
    $\bx\dt\ba= \text{circ}(\bx)\ba =\begin{bmatrix}
       1 & 0 & -1\\
       -1 & 1 & 0\\
       0 & -1 & 1
    \end{bmatrix} \begin{bmatrix}
        \frac{2}{3} \\
         \frac{1}{3}\\
         0
    \end{bmatrix}=\begin{bmatrix}
        \frac{2}{3} \\
         -\frac{1}{3}\\
         -\frac{1}{3}
   \end{bmatrix}$,  and 
    $\bx\dt\ba\dt\bx= \text{circ}(\bx\dt\ba)\bx=\begin{bmatrix}
       \frac{2}{3} & -\frac{1}{3} & -\frac{1}{3}\\
       -\frac{1}{3} & \frac{2}{3} & -\frac{1}{3}\\
       -\frac{1}{3} & -\frac{1}{3} & \frac{2}{3}
    \end{bmatrix} \begin{bmatrix}
        1 \\
         -1\\
         0\end{bmatrix}=\begin{bmatrix}
        1 \\
         -1\\
         0\end{bmatrix}=\bx$.
\end{center}
Hence $\bx$ is the annihilator $(\bb,\bc)$-inverse of $\ba$.
\end{example}

The uniqueness of annihilator $(\bb,\bc)$-inverse easily follows from left or right annihilator inverses. Using Lemma \ref{lem2.18} and Lemma \ref{lem2.19}, we state the next result concerning for annihilator $(\bb,\bc)$-inverse.

\begin{theorem}
Let $\ba, \bb, \bc, \bx\in \rn$. Then the following  statements are  are equivalent:
\begin{enumerate}
 \setlength{\parskip}{0pt}
    \item[(i)] $\bx\dt\ba\dt\bx=\bx,~\textup{lann}(\bx)=\textup{lann}(\bb),~\textup{rann}(\bx)=\textup{lann}(\bc)$.
    \item[(ii)] $\bx\dt\ba\dt\bb=\bb,~\bc\dt\ba\dt\bx=\bc,~\textup{lann}(\bb)\subseteq\textup{lann}(\bx),~\textup{rann}(\bc)\subseteq\textup{rann}(\bx)$.
\end{enumerate}
\end{theorem}

The group inverses on $\rn$ is defined as follows. 

\begin{definition}
Let $\ba\in \rn$. An element $\bx\in\rn$ is called a group inverse of $\ba$ if it satisfies 
\begin{center}
  $\ba\dt\bx\dt\ba=\ba,~\bx\dt\ba\dt\bx=\bx,~\ba\dt\bx=\bx\dt\ba$.
\end{center}
\end{definition}
Note that the condition $\ba\dt\bx=\bx\dt\ba$ is always satisfied, since $\rn$ is a commutative ring.
The group inverse of an element $\ba$ is unique if exists. We denote the group inverse of an element $\ba\in\rn$ by $\ba^{\#}$. An element $\ba$ is called group invertible if $\ba^{\#}$ exists. We next give an example of the group inverse of $\textbf{a} \in \rn$.

\begin{example}\rm
Let $\ba=(1, 0, -1 )^T\in\mathbb{R}^3$ and consider $\bx=\left(1/3, -1/3, 0\right)^T\in\mathbb{R}^3$. Then we can verify that  
\begin{center}
    $\bx\dt\ba=\ba\dt\bx= \text{circ}(\ba)\bx =\begin{bmatrix}
       1  &  -1  &    0\\
     0   &  1  &  -1\\
    -1  &   0  &   1
    \end{bmatrix} \begin{bmatrix}
        \frac{1}{3} \\
         -\frac{1}{3}\\
         0
    \end{bmatrix}=\begin{bmatrix}
        \frac{2}{3} \\
         -\frac{1}{3}\\
         -\frac{1}{3}
   \end{bmatrix}$,  
\end{center}
\begin{center}
    $\bx\dt\ba\dt\bx= \text{circ}(\bx\dt\ba)\bx=\begin{bmatrix}
       \frac{2}{3} & -\frac{1}{3} & -\frac{1}{3}\\
       -\frac{1}{3} & \frac{2}{3} & -\frac{1}{3}\\
       -\frac{1}{3} & -\frac{1}{3} & \frac{2}{3}
    \end{bmatrix} \begin{bmatrix}
        \frac{1}{3} \\
         -\frac{1}{3}\\
         0\end{bmatrix}=\begin{bmatrix}
        \frac{1}{3} \\
         -\frac{1}{3}\\
         0\end{bmatrix}=\bx$, and 
     $\ba\dt\bx\dt\ba= \text{circ}(\bx\dt\ba)\bx=\begin{bmatrix}
       \frac{2}{3} & -\frac{1}{3} & -\frac{1}{3}\\
       -\frac{1}{3} & \frac{2}{3} & -\frac{1}{3}\\
       -\frac{1}{3} & -\frac{1}{3} & \frac{2}{3}
    \end{bmatrix} \begin{bmatrix}
        1\\
         0\\
         -1\end{bmatrix}=\begin{bmatrix}
        1 \\
        0\\
         -1\end{bmatrix}=\ba$.
\end{center}
Hence $\bx$ is the group inverse of $\ba$.
\end{example}

The existence of group inverse is discussed in the theorem below: 

\begin{theorem}
Let $\ba\in \rn$. Then $\ba$ is group invertible if and only if $\ba\in\ba^2\dt\rn$. Moreover, if $\ba=\bx\dt\ba^2=\ba^2\dt\by$, for some $\bx, \by\in\rn$, then $\ba^{\#}=\bx\dt\ba\dt\by=\bx^2\dt\ba=\ba\dt\by^2$.
\end{theorem}
\begin{proof}
Let $\ba$ be group invertible. Then $\ba=\ba\dt\bx\dt\ba=\ba^2\dt\bx=\bx\dt\ba^2\in\ba^2\dt\rn$. Conversely, let $\ba=\bx\dt\ba^2=\ba^2\dt\by$, for some $\bx, \by\in\rn$. Then $\bx\dt\ba\dt\by=\bx\dt\ba^2\dt\by\dt\by=\ba\dt\by^2$ and $\bx\dt\ba\dt\by=\ba\dt\bx\dt\ba^2\dt\by=\bx^2\dt\ba$. Next we shall prove $\bx\dt\ba\dt\by$ is the group inverse of $\ba$. Now $\ba\dt(\bx\dt\ba\dt\by)\dt\ba=\bx\dt\ba^2\dt\ba\dt\by=\ba^2\dt\by=\ba$, $\bx\dt\ba\dt\by\dt(\ba)\dt\bx\dt\ba\dt\by=\bx\dt\ba^3\dt\by\dt\bx\dt\by=\ba^2\dt\by\dt\bx\dt\by=\ba\dt\bx\dt\by=\bx\dt\ba\dt\by$, and the commutative property holds trivially. Thus $\ba^{\#}=\bx\dt\ba\dt\by$.
\end{proof}
\begin{remark}
 If $\rn$ is any commutative ring with unity, then the group inverse of an element $\ba\in \rn$ exists if and only if $\ba\in\ba^2 \rn$.
\end{remark}

\section{Tensor over a non-commutative ring}
\subsection{Notation and Definitions}
Let $ \mathbb{R}^{n_1 \times n_2 \times n_3 \cdots \times n_p}$  be the set of order $p$ and dimension $n_1 \times \cdots \times n_p$ tensors over the real
field $\mathbb{R}$.  Let $ \mathcal{A} = (a_{i_1 ...i_p}) \in \mathbb{R}^{n_1 \times \cdots \times n_p}, p>1$. For $ i = 1,...,n_p $, denote by $ \mc{A}_i \in \mathbb{R}^{n_1 \times \cdots \times n_{p-1}} $, the tensor whose order is $(p - 1)$ and is created by holding the $p$th index of $\mc{A}$ fixed at $i$, which we called the frontal slices of the tensor $\mc{A}$. The generalization of matrix rows and columns are called fibers. Specifically, fixing all the indexes of a tensor $\mc{A}$ except one index. Now, define $ \text{unfold}(.)$ to take an $n_1 \times n_2 \times \cdots \times n_p $ tensor \cite{Braman10} and return an $ n_1 n_p \times n_2 \times \cdots \times n_{p-1} $ block tensor in the following way:\\
\begin{equation*}
\text{unfold}( \mc{A} ) =
\begin{bmatrix}
\mc{A}_1\\
\mc{A}_2\\
\vdots\\
\mc{A}_{n_p}
\end{bmatrix},
\end{equation*}
and $ \text{fold}()$ is the inverse operation \cite{Braman10}, which takes an 
$n_1 n_p \times n_2 \times \cdots \times n_{p-1} $ block tensor and returns an $n_1 \times n_2 \times \cdots \times n_p $ tensor. Then, $\text{fold}( \text{unfold}(\mc{A})) = \mc{A}.$ Now, one can easily see  that, 
\begin{equation*}
 \text{circ}( \text{unfold}(\mc{A} )) = 
   \begin{bmatrix}
\mc{A}_1 & \mc{A}_{n_p} & \cdots & \mc{A}_2\\
\mc{A}_2 & \mc{A}_1    &  \cdots & \mc{A}_3\\
\vdots  &  \vdots  & \vdots  &    \vdots\\
\mc{A}_{n_p}& \mc{A}_{n_{p-1}}& \cdots & \mc{A}_1
 \end{bmatrix}. 
\end{equation*}

\begin{definition}\label{TransDefn} {\cite{Martin13}}
Let $\mc{A} \in \mathbb{R}^{n\times n \times n_3\times \cdots\times n_p}.$ The transpose of  $\mc{A}$ is denoted by $\mc{A}^T \in \mathbb{R}^{n\times n \times n_3\times \cdots\times n_p}$ is defined by the tensor transposing each $\mc{A}_i$, for $i = 1, 2, \cdots, n$ along with reversing the order of the $\mc{A}_i$ from $2$ through $n_p$.
\end{definition}

We collect some useful definitions from \cite{Martin13} as follows. A tensor $\mc{D} \in \mathbb{R}^{n\times n \times n_3\times \cdots\times n_p} $ is called a $f$-{\it diagonal
   tensor} if each frontal slice is diagonal. Similarly, a tensor is a $f$-upper triangular or $f$-lower triangular if each frontal slice is upper or lower triangular, respectively. Likewise, a tensor is called the identity tensor if each frontal slice is the identity matrix and all other frontal slices are zeros. The tensor whose entries are all zero is denoted by $\mc{O}$.

Now we construct a ring over the tensor space $\mathbb{R}^{n\times n \times n_3\times \cdots\times n_p}$.
\begin{theorem}
 Let $ \mathbb{R}^{n \times n \times n_3 \cdots \times n_p}$  be  a $p$-th order tensor over a field $\mathbb{R}$ with binary operations ($+,*$) (for addition and  multiplication), 
 defined as 
\begin{equation*}
\mc{A} * \mc{B} = \text{fold}( \text{circ}( \text{unfold}(\mc{A})) *  \text{unfold}(\mc{B})), \mbox{ and }
\mc{A} + \mc{B} =\mc{A}_{i_1...i_p} + \mc{B}_{i_1...i_p}, ~~\forall \mc{A}, \mc{B} \in \mathbb{R}^{n\times n \times n_3\times \cdots\times n_p}.
\end{equation*}
Then $\mathbb{R}^{n\times n \times n_3\times \cdots\times n_p}$ is a ring with unity.
\end{theorem}
\begin{proof}
It is easy to show that $(\mathbb{R}^{n\times n \times n_3\times \cdots\times n_p}, +)$ is an group and it is straightforward. Using the definition of $*$ and its associative properties, we can show that: 
\begin{eqnarray*}
\mc{A}*(\mc{B}+ \mc{C}) &=& \text{fold}( \text{circ}( \text{unfold}(\mc{A})) *  \text{unfold}(\mc{B}+\mc{C}))\\
&=& \text{fold}( \text{circ}( \text{unfold}(\mc{A})) * [ \text{unfold}(\mc{B})+ \text{unfold}(\mc{C})])\\
&=& \text{fold}( \text{circ}( \text{unfold}(\mc{A})) * \text{unfold}(\mc{B}))+ \text{fold}( \text{circ}( \text{unfold}(\mc{A})) * \text{unfold}(\mc{C}))\\
&=&\mc{A}*\mc{B}+ \mc{A}*\mc{C}.
\end{eqnarray*}
Similarly, we can prove that, $(\mc{B}+\mc{C})*\mc{A} = \mc{B}*\mc{A}+\mc{C}*\mc{A}$.
\end{proof}

\begin{remark}
One can easily see that $\left(\mathbb{R}^{n\times n \times n_3\times \cdots\times n_p}, *, +\right)$ is a non-commutative ring. For simplicity, we use the notation $\mathbb{R}^{n\times n \times n_3\times \cdots\times n_p}$ for $\left(\mathbb{R}^{n\times n \times n_3\times \cdots\times n_p}, *, +\right)$  to facilitate the presentation.
\end{remark}

We next give an example of the non-commutative ring of tensors. 

\begin{example}
Let $ \mathcal{A}, \mathcal{B} \in \mathbb{R}^{3 \times 3 \times 2}$  
 with frontal slices 
\begin{eqnarray*}
\mc{A}_{(1)} =
    \begin{pmatrix}
    1 & 2 & 3\\
    0 & 0 & 0\\
     0 & 0 & 0
    \end{pmatrix},~
\mc{A}_{(2)} =
    \begin{pmatrix}
     1 & 2 & 1\\
     0 & 2 & 0\\
      3 & 1 & 0
    \end{pmatrix},~~
\mc{B}_{(1)} =
    \begin{pmatrix}
    1 & 1 & 0\\
    0 & 3 & 0\\
     0 & 1 & 0
    \end{pmatrix},~
\mc{B}_{(2)} =
    \begin{pmatrix}
     0 & 1 & 0\\
     0 & 5 & 0\\
     1 & 1 & 0
    \end{pmatrix}.
    \end{eqnarray*}
Then we can see that 
\begin{eqnarray*}
fold\left( 
\begin{bmatrix}
     2 & 22 & 0\\
     0 & 10 & 0\\
     0 & 8 & 0\\\hline 
     4 & 22 & 0\\
     0 & 6 & 0\\
     3 & 6 & 0
\end{bmatrix}
\right)
= \mc{A}*\mc{B}\neq  \mc{B}*\mc{A}= 
fold\left( 
\begin{bmatrix}
     1 & 4 & 3\\
     0 & 10 & 0\\
     1 & 4 & 1\\\hline 
     1 & 4 & 1\\
     0 & 6 & 0\\
     1 & 4 & 3
\end{bmatrix}
\right).
\end{eqnarray*}
\end{example}

\subsection{Computation of Generalized Inverses}
We now introduce the definition of $\{i\}$-inverses $(i = 1, 2, 3, 4)$ and the Moore-Penrose inverse of tensors ($t$-product) over $\rt$.
 \begin{definition}\label{defgi}
For any tensor $\mc{A} \in \rt,$  consider the following equations in $\mc{Z} \in \rt:$
\begin{eqnarray*}
&&(1)~\mc{A}*\mc{Z}*\mc{A} = \mc{A}, \hspace{1cm}
(3)~(\mc{A}*\mc{Z})^T = \mc{A}*\mc{Z},\\
&&(2)~\mc{Z}*\mc{A}*\mc{Z} = \mc{Z},\hspace{1cm}
(4)~(\mc{Z}*\mc{A})^T = \mc{Z}*\mc{A}.
\end{eqnarray*}
Then $\mc{Z}$ is called
\begin{enumerate}
 \setlength{\parskip}{0pt}
\item[(a)] 
an inner inverse  of $\mc{A}$ if it satisfies $(1)$ and is denoted by $\mc{A}^{(1)}$;
\item[(b)]  an outer inverse of $\mc{A}$ if it satisfies $(1)$ and $(2)$, which is denoted by $\mc{A}^{(1,2)}$;
\item[(c)] a $\{1,3\}$ inverse of $\mc{A}$ if it satisfies $(1)$ and $(3)$, which is denoted by $\mc{A}^{(1,3)}$;
\item[(d)]  a $\{1,4\}$ inverse of $\mc{A}$ if it satisfies $(1)$ and $(4)$, which is  denoted by $\mc{A}^{(1,4)}$;
\item[(e)] the Moore-Penrose inverse of $\mc{A}$ if it satisfies all four conditions $[(1)-(4)]$, which is denoted by $\mc{A}^{\dagger}.$
\end{enumerate}
\end{definition}

It is worth mentioning that the discrete Fourier Transform plays significantly role for the product of two tensors. For instance, if $\textbf{a} = (a_1, a_2, \cdots, a_n)^T$ is $n \times 1$ vector, then $F_n circ(\textbf{a})F_n^*$ is diagonal, where $F_n$ is the $n \times n$ discrete Fourier transform (DFT) matrix, and $F_n^*$ is its conjugate transpose. To compute this diagonal, the fast Fourier transform (FFT) is used \cite{Golub96Book} as follows. 
\begin{equation*}
    F_n circ(\textbf{a})F_n^*= \textnormal{\bf fft}(\textbf{a})
\end{equation*}
\begin{algorithm}[hbt!]
  \caption{Computation of transpose of a tensor $\mc{A}$} \label{AlgoTran}
  \begin{algorithmic}[1]
    \Procedure{TRANSPOSE}{$\mc{A}$}
    \State \textbf{Input} $p$, $n,$ $n_3,\ldots,n_p$ and the tensor $\mathcal{A}\in \rt$.
      \For{$i \gets 1$ to $p$} 
      \State $\mc{A}=\textup{fft}(\mc{A}, [~ ], i);$
      \EndFor
      \State $N=n_3n_4\cdots n_p$
      \For{$i \gets 1$ to $N$} 
       \State $\mc{Z}(:, :, i ) = \textup{transpose}(\mc{A}(:, :, i ));$
      \EndFor
      \For{$i \gets p$ to $1$} 
      \State $\mc{B} \leftarrow \textup{ifft}(\mc{Z}, [~ ], i);$
      \EndFor
      \State \textbf{return} $\mc{B}$\Comment{$\mc{B}$ is the transpose of $\mc{A}$}
    \EndProcedure
  \end{algorithmic}
\end{algorithm}

Further, the authors of \cite{Kilmer13, kilmer11, Martin13, BramanK07} utilized it for tensors. Suppose  $\mc{A} \in \mathbb{R}^{n\times n \times n_3}$, then $\text{circ}(\text{unfold}(\mc{A}))$ is a block circulant matrix with each  $ A_i \in \mathbb{R}^{n \times n},$ for $1\leq i \leq n_3$.  Then 
\begin{equation*}
    (F_{n_3} \otimes I_{n})\cdot 
    \begin{pmatrix}
     {A}_1 & {A}_{n_3} & \cdots & {A}_2\\
{A}_2 & {A}_1    &  \cdots & \mc{A}_3\\
\vdots  &  \vdots  & \ddots  &    \vdots\\
A_{n_3}& {A}_{n_3-1}& \cdots & {A}_1
    \end{pmatrix} \cdot (F_{n_3}^* \otimes I_{n})=
    \begin{pmatrix}
     {D}_1 &  &  & \\
     & {D}_2 &  & \\
     & &  \ddots &\\ 
    & &  & {D}_{n_3}
    \end{pmatrix},
\end{equation*}
where $F_{n_3}$ is the $n_3\times n_3$ DFT matrix, $F_{n_3}^*$ is the conjugate transpose of $F_{n_3}$, and `$\cdot$' is the standard matrix multiplication. Here $ D_i \in \mathbb{R}^{n \times n},$ for $1\leq i \leq n_3$.  Similarly for $\mc{A} \in \rt,$ we can write
\begin{equation}\label{los}
    (F_{n_p} \otimes F_{n_{p-1}} \otimes \dots F_{n_3} \otimes I_{n})\cdot \text{circ}(\text{unfold}(\mc{A}) \cdot (F^*_{n_p} \otimes F^*_{n_{p-1}} \otimes \dots F^*_{n_3} \otimes I_{n})= \begin{pmatrix}
     {\Sigma}_1 &  &  & \\
     & {\Sigma}_2 &  & \\
     & &  \ddots &\\ 
    & &  & {\Sigma}_{\rho}
    \end{pmatrix}
\end{equation}
is a block diagonal matrix with $\rho$ blocks each of size $n\times n$, where $\rho=n_3n_4\cdots n_p$. Hence computation of tensors via the Fourier domain are obtained by systematically reorganizing the tensor into a matrix. Then, the benefits of the matrix computation results will be utilized for tensor computation. So once the matrix computation is performed on the Fourier domain we apply $(F^*_{n_p} \otimes F^*_{n_{p-1}} \otimes \dots F^*_{n_3} \otimes I_{n})$ to the left and $(F_{n_p} \otimes F_{n_{p-1}} \otimes \dots F_{n_3} \otimes I_{n})$ 
 to the right of each of the block diagonal matrices in Eq. \eqref{los}.  Folding up those
results takes us back into the appropriate sized tensor results. In view of this representation, we next present the definition of the symmetric positive definite tensor which was introduced earlier in \cite{Kilmer13}. A tensor $\mc{A} \in \rt$ is called symmetric positive definite tensor if all the $\Sigma_i$ for $i=1,2, \cdots \rho$, are hermitian positive definite.

The fast Fourier transform is utilized in Algorithm-\ref{AlgoTran} and Algorithm-\ref{AlgoGen} for computing transpose of a tensor and inner inverses of tensors. Following the Definition \ref{TransDefn}, the ``transpose'' function is used in Algorithm-\ref{AlgoTran} on line-8 to compute the transpose of matrices in the Fourier domain. Similarly, the functions (i.e., rank, rref) are used in Algorithm-\ref{AlgoGen} on line-8 and line-9 to compute rank and reduced row-echelon form of matrices in the Fourier domain, respectively.  Here, our purpose is not to compare our tensor-based approach to other methods, but rather to contribute to the class of algorithms used for this purpose.

\begin{definition}
The left and right ideals generated by $\mc{A} \in \rt$ are respectively defined by
\begin{equation*}
   \rt*\mc{A}=\{ \mc{Z}*\mc{A} ~:\mc{Z}\in \rt \} ~\mbox{ and ~} \mc{A}*\rt=\{ \mc{A}* \mc{Z} ~:\mc{Z}\in \rt \}.
\end{equation*}
\end{definition}
\begin{definition}
The right annihilator of $\mc{A} \in \rt$ denoted by $\textup{rann}(\mc{A})$ and 
left annihilator of $\mc{A}$ denoted by $\textup{lann}(\mc{A})$ are defined respectively by
\begin{equation*}
     \textup{rann}(\mc{A})=\{\mc{X}\in \rt:~\mc{A} * \mc{X}=\mc{O} \}~~\textnormal{and}~~ \textup{lann}(\mc{A})=\{\mc{X}\in \rt:~ \mc{X} *\mc{A}=\mc{O}  \}.
\end{equation*}
\end{definition}

\begin{algorithm}[hbt!]
  \caption{Computation of inner inverse of a tensor $\mc{A}$} \label{AlgoGen}
  \begin{algorithmic}[1]
    \Procedure{IINV}{$\mc{A}$}
    \State \textbf{Input} $p$, $n,$ $n_3,\ldots,n_p$ and the tensor $\mathcal{A}\in \rt$.
      \For{$i \gets 1$ to $p$} 
      \State $\mc{A}=\textup{fft}(\mc{A}, [~ ], i);$
      \EndFor
      \State $N=n_3n_4\cdots n_p$
      \For{$i \gets 1$ to $N$} 
        \State $r\leftarrow \textup{rank}({\mathcal{A}(:,:,i)}).$
		
		\State $\mc{B} \leftarrow \textup{rref}([A(:,:,i), I_{r}(:,:,i)])$.
		\State $\mc{G}\leftarrow $ last $r$ columns of $B(:,:,i)$.
		\State Find a permutation matrix $\mc{P}(:,:,i)$ such that  $\mc{G}(:,:,i)\mc{A}(:,:,i)\mc{P}(:,:,i) = \bmatrix \c{I}_r(:,:,i) & \mc{K}(:,:,i) \\ \mc{O}(:,:,i) & \mc{O}(:,:,i)\endbmatrix.$ 
		\State Define an arbitrary  $\mc{L}(:,:,i)$ such that 
		$\mc{B}(:,:,i)= \bmatrix \c{I}_r(:,:,i) & \mc{O}(:,:,i) \\ \mc{O}(:,:,i) & \mc{L}(:,:,i)\endbmatrix.$
		\State $\mathcal{Z}(:,:,i)=\mc{P}(:,:,i)\mc{B}(:,:,i)  \mc{G}(:,:,i).$
      \EndFor
      \For{$i \gets p$ to $1$} 
      \State $\mc{X}=\textup{ifft}(\mc{Z}, [~ ], i);$
      \EndFor
      \State \textbf{return} $\mc{X}$\Comment{$\mc{X}$ is the inner inverse of $\mc{A}$}
    \EndProcedure
  \end{algorithmic}
\end{algorithm}

The annihilators of $\mc{A} \in \rt$ satisfy the following properties. 
\begin{proposition} \label{pro1}
Let $\mc{A}$, $\mc{B} \in \rt$. Then  $\mc{A}*\rt\subseteq \mc{B}*\rt$ 
if and only if  $\textup{lann}(\mc{B}) \subseteq \textup{lann}(\mc{A})$.
\end{proposition}

\begin{proof}
  Let  $\mc{U}\in\textup{lann}(\mc{B})$.  Then $\mc{U}*\mc{B} = 0$ for some $\mc{U} \in \rt$. From $\mc{A}*\rt \subseteq \mc{B}*\rt$, we get $\mc{A} = \mc{B}*\mc{V}$  for some $\mc{V} \in \rt$. Now  $\mc{U}*\mc{A} = \mc{U}*\mc{B}*\mc{V} = \mc{O}$. Thus $\mc{U}\in \textup{lann}(\mc{A})$ and hence $\textup{lann}(\mc{B}) \subseteq \textup{lann}(\mc{A})$.\\
Conversely from   $(\mc{I}-\mc{B}*\mc{B}^{(1)})*\mc{B} = \mc{B}-\mc{B}*\mc{B}^{(1)}*\mc{B}=\mc{O}$ and  $\textup{lann}(\mc{B}) \subseteq \textup{lann}(\mc{A})$, we have  $(\mc{I} -\mc{B}*\mc{B}^{(1)} )*\mc{A} = \mc{O}$. Thus $\mc{A}  = \mc{B}*\mc{B}^{(1)} *\mc{A}$. Now let $\mc{S}\in\mc{A}*\rt$. Then $\mc{S}=\mc{A}*\mc{U}=\mc{B}*\mc{B}^{(1)} *\mc{A}*\mc{U}=\mc{B}*\mc{T}$ for some $\mc{T}=\mc{B}^{(1)} *\mc{A}*\mc{U}\in\rt$. Hence $\mc{A}*\rt \subseteq \mc{B}*\rt$.
\end{proof}
Using similar arguments, we can show the following result for the right anniihilator.
\begin{proposition}\label{prop2}
Let $\mc{C}, \mc{D} \in \rt $. Then $\rt *\mc{C} \subseteq \rt *\mc{D}$ if and only if  $\textup{rann}(\mc{D}) \subseteq \textup{rann}(\mc{C})$.
\end{proposition}

The existence of solution of tensor equation through one-inverse is discussed in the following theorem, which can be easily proved. 
\begin{theorem}
Let $\rt$ be an associative ring with identity $\mc{I}$. Let $\mc{A} \in \rt$ and $\mc{X} \in \rt$. Then the following statements are equivalent:
\begin{enumerate}
    \item[(i)] $\mc{X}*\mc{B}$ is a solution of the multilinear system $\mc{A}*\mc{Y}=\mc{B}$ whenever $\mc{B}\in \mc{A}*\rt$.
    \item[(ii)] $\mc{A}*\mc{X}*\mc{A}=\mc{A}$.
\end{enumerate}
\end{theorem}

However, if $\rt$ is an associative ring without identity then the above theorem does not hold in general, as shown by the next example.
\begin{example}\rm
Consider an associative ring of all $3\times 3\times 3$ real tensor. The first, second and third frontal slices are
\begin{eqnarray*}
    \begin{pmatrix}
    a & a & a\\
    0 & 0 & 0\\
     0 & 0 & 0
    \end{pmatrix},~
    \begin{pmatrix}
     0 & 0 & 0\\
     0 & 0 & 0\\
      0 & 0 & 0
    \end{pmatrix},~~
    \begin{pmatrix}
    0 & 0 & 0\\
    0 & 0 & 0\\
     0 & 0 & 0
    \end{pmatrix}.
    \end{eqnarray*}
 Consider $\mc{A}$ in the above set of frontal slice. But $\mc{A}^\dg$ is not in the above sets, i.e., The first, second and third frontal slices of $\mc{A}^\dg$ are:
   \begin{eqnarray*}
    \begin{pmatrix}
    \frac{1}{3a} & 0 & 0\\
    \frac{1}{3a} & 0 & 0\\
     \frac{1}{3a} & 0 & 0
    \end{pmatrix},~
    \begin{pmatrix}
     0 & 0 & 0\\
     0 & 0 & 0\\
      0 & 0 & 0
    \end{pmatrix},~~
    \begin{pmatrix}
    0 & 0 & 0\\
    0 & 0 & 0\\
     0 & 0 & 0
    \end{pmatrix}.
    \end{eqnarray*}
   \end{example}

We now discuss the reverse order of $\mc{A}^{(1)}$ and $ \mc{A}^{(2)}$ inverses.

\begin{proposition}
Let $\mc{A},\mc{B}\in \rt$. Then the following are true.
\begin{enumerate}
    \item[(i)] If $\mc{A}^{(1)}*\mc{A}*\mc{B}=\mc{B}*\mc{A}^{(1)}*\mc{A}$, then $        (\mc{A}*\mc{B})^{(1)}=\mc{B}^{(1)}*\mc{A}^{(1)}$
    \item[(ii)] If $\mc{B}*\mc{B}^{(2)}*\mc{A}^{(2)}=\mc{A}^{(2)}*\mc{B}*\mc{B}^{(2)}$, then $        (\mc{A}*\mc{B})^{(2)}=\mc{B}^{(2)}*\mc{A}^{(2)}.
    $
\end{enumerate}
\end{proposition}
\begin{proof}
$(i)$ Let $\mc{A}*\mc{A}^{(1)}*\mc{B}=\mc{B}*\mc{A}*\mc{A}^{(1)}$. Then the result is follows from 
\begin{eqnarray*}
\mc{A}*\mc{B}*\mc{B}^{(1)}*\mc{A}^{(1)}*\mc{A}*\mc{B}&=&\mc{A}*\mc{B}*\mc{B}^{(1)}*\mc{B}*\mc{A}^{(1)}*\mc{A}=\mc{A}*\mc{B}*\mc{A}^{(1)}*\mc{A}=\mc{A}*\mc{A}^{(1)}*\mc{A}*\mc{B}=\mc{A}*\mc{B}.
\end{eqnarray*}
$(ii)$ Let $\mc{B}*\mc{B}^{(2)}*\mc{A}^{(2)}=\mc{A}^{(2)}*\mc{B}*\mc{B}^{(2)}$. Then 
\begin{eqnarray*}
\mc{B}^{(2)}*\mc{A}^{(2)}*\mc{A}*\mc{B}*\mc{B}^{(2)}*\mc{A}^{(2)}&=&\mc{B}^{(2)}*\mc{A}^{(2)}*\mc{A}*\mc{A}^{(2)}*\mc{B}*\mc{B}^{(2)}=\mc{B}^{(2)}*\mc{A}^{(2)}*\mc{B}*\mc{B}^{(2)}\\
&=&\mc{B}^{(2)}*\mc{B}*\mc{B}^{(2)}*\mc{A}^{(2)}=\mc{B}^{(2)}*\mc{A}^{(2)}.
\end{eqnarray*}
\end{proof}

The next result is for reflexive generalized inverse of a tensor
$\mc{A}\in \rt$, which can be easily proved.

\begin{lemma}\label{12}
Let $\mc{A},\mc{X}\in\rt$ If $\mc{Y}, \mc{Z} \in  \mc{A}{\{1\}}$ and $\mc{X} = \mc{Y}*\mc{A}*\mc{Z}$, then $\mc{X} \in  \mc{A}{\{1, 2\}}$.
\end{lemma}

The next results are discussed for $\{1,3\}$ and $\{1, 4\}$ inverses of a tensor $\mc{A} \in \rt$, which can be proved easily.

\begin{lemma}
 Let $\mc{A}\in \rt$. Then the following are holds.
\begin{enumerate}
    \item[(i)] If  $\mc{A}=\mc{X}*\mc{A}^T*\mc{A}$ for some $\mc{X}\in \rt$, then $\mc{X}^T$ is a $\{1, 3\}$-inverse of $\mc{A}$.
    \item[(ii)]  If $\mc{A}=\mc{A}*\mc{A}^T*\mc{Y}$ for some $\mc{Y}\in R$, then $\mc{Y}^T$ is a $\{1, 4\}$-inverse of $\mc{A}$. 
\end{enumerate}
\end{lemma}

\begin{theorem}
Let $\mc{A},\mc{B}\in\rt$. Then the following statements are equivalent:
\begin{enumerate}
    \item[(i)] $\mc{B} \in \mc{A}\{1,3\}$.
 \item[(ii)] $ \mc{A}^T * \mc{A}* \mc{B} =\mc{A}^T$.
 \item[(iii)] $\mc{A} * \mc{B}= \mc{A}* (\mc{A}^T* \mc{A})^{(1)}*
 \mc{A}^T.$
\end{enumerate}
\end{theorem}
\begin{proof}
$(i)\Rightarrow(ii)$ Let $\mc{B} \in \mc{A}\{1,3\}$. Then  $\mc{A} * \mc{B}*
\mc{A}=\mc{A}$  and $(\mc{A} * \mc{B})^T=\mc{A} * \mc{B}$. Now 
\begin{center}
  $\mc{A}^T= (\mc{A} * \mc{B}* \mc{A})^T =\mc{A}^T *(\mc{A}* \mc{B})^T=\mc{A}^T * \mc{A}*\mc{B}$.
\end{center}
$(ii)\Rightarrow (iii)$ Let $ \mc{A}^T * \mc{A}* \mc{B} =\mc{A}^T$. Pre-multiplying by $\mc{A}*(\mc{A}^T *
 \mc{A})^{(1)}$, we obtain
 \begin{center}
     $\mc{A}* (\mc{A}^T* \mc{A})^{(1)}*
 \mc{A}^T=\mc{A}* (\mc{A}^T* \mc{A})^{(1)}*\mc{A}^T * \mc{A}* \mc{B}= \mc{B}^T*\mc{A}^T*\mc{A}* \mc{B}=\mc{A}* \mc{B}$.
 \end{center}
$(iii)\Rightarrow (i)$ First we will show that $\rt*\mc{A}^T\subseteq\rt*\mc{A}^T*\mc{A}$. Let $\mc{Y}\in \rt*\mc{A}^T$. Then $\mc{Y}=\mc{A}^T*\mc{Z}$ for some $\mc{Z}\in\rt$. Now 
\begin{center}
    $\mc{Y}=\mc{A}^T*\mc{Z}=(\mc{A}*\mc{A}^{(1,3)}*\mc{A})^T*\mc{Z}=\mc{A}^T*\mc{A}*\mc{T}$ where $\mc{T}=\mc{A}^{(1,3)}*\mc{Z}\in\rt$.
\end{center}
Thus $\rt*\mc{A}^T\subseteq\rt*\mc{A}^T*\mc{A}$. From this condition we can show that $\mc{A} * (\mc{A}^T *
\mc{A})^{(1)}* \mc{A}^T * \mc{A}=\mc{A}$, so we have $\mc{A}* \mc{B}*\mc{A}=\mc{A}$.   Further, $(\mc{A}* \mc{B})^T=\mc{A} * (\mc{A} * \mc{A}^T)^{(1)}* \mc{A}^T=\mc{A}* \mc{B}.$
\end{proof}

Similarly, for $\{1,4\}$-inverse, we state the following result without proof.
\begin{theorem}
The following three conditions are equivalent:\\
 (i) $\mc{B} \in \mc{A}\{1,4\}$.\\
 (ii) $\mc{B} * \mc{A}* \mc{A}^T=\mc{A}^T$.\\
 (iii) $\mc{B} * \mc{A}= \mc{A}^T* (\mc{A}* \mc{A}^T)^{(1)}*
 \mc{A}.$
\end{theorem}

A sufficient condition for the reverse order law of
$\{1,4\}$-inverses and $\{1,3\}$-inverses of tensors is given in the next result.

\begin{theorem}\label{2.26}
The following conditions are true for any $\mc{A},\mc{B}\in\rt$, 
\begin{itemize}
    \item[(i)] If $\mc{A}^{(1,4)}* \mc{A}* \mc{B}* \mc{B}^T$ is symmetric, then $(\mc{A} * \mc{B})^{(1,4)}=\mc{B}^{(1,4)} * \mc{A}^{(1,4)}.$ 
\item[(ii)] If $\mc{A}* \mc{A}^{(1,3)}*  \mc{B}^T* \mc{B}$ is symmetric, then $(\mc{A} * \mc{B})^{(1,3)}=\mc{B}^{(1,3)} * \mc{A}^{(1,3)}.$
\end{itemize}
\end{theorem}

\begin{proof}
(i) Let  $(\mc{A}^{(1,4)}* \mc{A}* \mc{B}* \mc{B}^T)^T=\mc{A}^{(1,4)}* \mc{A}* \mc{B}* \mc{B}^T$. Then 
\begin{equation}\label{eq8}
\mc{A}^{(1,4)}* \mc{A} *\mc{B}* \mc{B}^T
= (\mc{A}^{(1,4)}* \mc{A} *\mc{B}* \mc{B}^T)^T
 = \mc{B}* \mc{B}^T*\mc{A}^{(1,4)} * \mc{A}.
\end{equation}
Using \eqref{eq8}, we obtain
\begin{eqnarray*}
\mc{A} *\mc{B} * \mc{B}^{(1,4)} * \mc{A}^{(1,4)} * \mc{A} * \mc{B}&=& \mc{A} *\mc{B} * \mc{B}^{(1,4)} * \mc{A}^{(1,4)} * \mc{A} * \mc{B}*(\mc{B}^{(1,4)}*\mc{B})^T\\
&=&\mc{A} *\mc{B} * \mc{B}^{(1,4)} * \mc{A}^{(1,4)} * \mc{A} * \mc{B}*\mc{B}^T*(\mc{B}^{(1,4)} )^T\\
&=&\mc{A}*\mc{B}*\mc{B}^T*\mc{A}^{(1,4)} * \mc{A} * (\mc{B}^{(1,4)} )^T=\mc{A} * \mc{B} *\mc{B}^T*(\mc{B}^{(1,4)})^T\\
&=&\mc{A} *\mc{B},\mbox{ and }
\end{eqnarray*}
\begin{eqnarray*}
(\mc{B}^{(1,4)} * \mc{A}^{(1,4)} * \mc{A}*
\mc{B})^T&=& \mc{B}^T*\mc{A}^{(1,4)}*\mc{A}* 
(\mc{B}^{(1,4)})^T=\mc{B}^{(1,4)} *\mc{B}*  \mc{B}^T* \mc{A}^{(1,4)} * \mc{A}* 
(\mc{B}^{(1,4)})^T \\
&=& \mc{B}^{(1,4)}*  \mc{A}^{(1,4)} * \mc{A}* \mc{B} * \mc{B}^T *
(\mc{B}^{(1,4)})^T=\mc{B}^{(1,4)}*  \mc{A}^{(1,4)} * \mc{A}* \mc{B}.
  \end{eqnarray*}
 Hence  $\mc{B}^{(1,4)} * \mc{A}^{(1,4)}$ is an $\{1,4\}$-inverse of $\mc{A} * \mc{B}$.
Similarly, one can prove the part $(ii)$.
\end{proof}

The conditions of the above theorem are sufficient but not necessary for the reverse order law. 

\begin{example}
Let $ \mathcal{A}, \mathcal{B} \in \mathbb{R}^{2\times 2 \times 3}$  
 with frontal slices
\begin{eqnarray*}
\mc{A}_{(1)} =
    \begin{pmatrix}
    1 & 3 \\
    0 & 0
    \end{pmatrix},~
\mc{A}_{(2)} =
    \begin{pmatrix}
     2 &  4 \\
     0 & 0
    \end{pmatrix},~~
    \mc{A}_{(3)} =
    \begin{pmatrix}
     5 &  1\\
     0 & 0
    \end{pmatrix},~~
\mc{B}_{(1)} =
    \begin{pmatrix}
    1 & 0\\
    0 & 1
    \end{pmatrix},~
\mc{B}_{(2)}=
    \begin{pmatrix}
     0 & 0\\
     0 & 0
    \end{pmatrix},~~
    \mc{B}_{(3)} =
    \begin{pmatrix}
    0 & 0\\
    0 & 0
    \end{pmatrix}.~
    \end{eqnarray*}
    We can verify that    
\begin{eqnarray*}
\mc{A}^{(1,4)}=
fold\left( 
\begin{bmatrix}
     3/80 & 0\\
     0 & -1/16\\\hline
     -1/16 & 0\\ 
     0 & 11/80\\\hline 
     7/80 & 0\\
     0 & -1/80\\
\end{bmatrix}
\right),~\mc{B}^{(1,4)}
= fold\left( 
\begin{bmatrix}
     1 & 0\\
     0 & 1\\\hline
     0 & 0\\ 
     0 & 0\\\hline 
     0 & 0\\
     0 & 0\\
\end{bmatrix}
\right),~
\mc{B}^{(1,4)}*\mc{A}^{(1,4)} = fold\left( 
\begin{bmatrix}
      3/80 & 0\\
     0 & -1/16\\\hline
     -1/16 & 0\\ 
     0 & 11/80\\\hline 
     7/80 & 0\\
     0 & -1/80\\
\end{bmatrix}
\right) = (\mc{A}*\mc{B})^{(1,4)},
\end{eqnarray*}
and 
\begin{eqnarray*}
fold\left( 
\begin{bmatrix}
     -1/10 & 2/5\\
     0 & 0\\\hline
     9/20 & 1/20\\ 
     0 & 0\\\hline 
     3/20 & 1/20\\
     0 & 0\\
\end{bmatrix}
\right) = \mc{A}^{(1,4)}* \mc{A}* \mc{B}* \mc{B}^T \neq  (\mc{A}^{(1,4)}* \mc{A}* \mc{B}* \mc{B}^T)^T
= fold\left( 
\begin{bmatrix}
          
     -1/10 & 0\\
     2/5 & 0\\\hline
     3/20 & 0\\ 
     1/20 & 0\\\hline 
     9/20 & 0\\
     1/20 & 0\\
\end{bmatrix}
\right)
\end{eqnarray*}
\end{example}

\begin{example}
Let $ \mathcal{A}, \mathcal{B} \in \mathbb{R}^{2\times 2 \times 3}$  
 with frontal slices 
\begin{eqnarray*}
\mc{A}_{(1)} =
    \begin{pmatrix}
    2 & 0 \\
    0 & 0
    \end{pmatrix},~
\mc{A}_{(2)} =
    \begin{pmatrix}
     3 &  0 \\
     0 & 0
    \end{pmatrix},~~
    \mc{A}_{(3)} =
    \begin{pmatrix}
     1 &  0\\
     0 & 0
    \end{pmatrix},~
\mc{B}_{(1)} =
    \begin{pmatrix}
    1 &  2\\
    3 & 0
    \end{pmatrix},~
\mc{B}_{(2)}=
    \begin{pmatrix}
     3 & 4\\
     5 & 0
    \end{pmatrix},~~
    \mc{B}_{(3)} =
    \begin{pmatrix}
    5 &  6\\
    6 & 0
    \end{pmatrix}.
    \end{eqnarray*}
We can verify that   
\begin{eqnarray*}
\mc{A}^{(1,3)}=
fold\left( 
\begin{bmatrix}
     1/18 & 0\\
     0 & 0\\\hline
     -5/18 & 0\\ 
     0 & 0\\\hline 
     7/18 & 0\\
     0 & 0\\
\end{bmatrix}
\right),~\mc{B}^{(1,3)}
= fold\left( 
\begin{bmatrix}
     0 & -3/14\\
     -5/36 & 37/168\\\hline
     0 & 3/14\\ 
     7/36 & -5/24\\\hline 
     0 & 1/14\\
     1/36 & -11/168\\
\end{bmatrix}
\right),
\mc{B}^{(1,3)}*\mc{A}^{(1,3)} = fold\left( 
\begin{bmatrix}
      0 & 0\\
     13/216 & 0\\\hline
     0 & 0\\ 
     13/216 & 0\\\hline 
     0 & 0\\
     -23/216 & 0\\
\end{bmatrix}
\right) = (\mc{A}*\mc{B})^{(1,3)},
\end{eqnarray*}
and 
\begin{eqnarray*}
fold\left( 
\begin{bmatrix}
     105 & 44\\
     0 & 0\\\hline
     86 & 32\\ 
     0 & 0\\\hline 
     86 & 32\\
     0 & 0\\
\end{bmatrix}
\right) = \mc{A}* \mc{A}^{(1,3)}* \mc{B}^T* \mc{B} \neq  (\mc{A}* \mc{A}^{(1,3)}* \mc{B}^T* \mc{B})^T
= fold\left( 
\begin{bmatrix}
    105 & 0\\
     44 & 0\\\hline
     86 & 0\\ 
     32 & 0\\\hline 
     86 & 0\\
     32 & 0\\
\end{bmatrix}
\right).
\end{eqnarray*}
\end{example}

\subsection{Results on the Moore-Penrose Inverse}
To discuss  a characterizations of the Moore-Penrose inverse, we first prove the following  auxiliary results which will be used for proving our main result of the section.

\begin{lemma}\label{lem3.17}
Let $\mc{P},\mc{Q}\in\rt$ be idempotent tensors. If  $\rt*\mc{P}\subseteq\rt*\mc{Q}$ and $\mc{Q}*\rt\subseteq\mc{P}*\rt$, then $\mc{P}=\mc{Q}$.
\end{lemma}
\begin{proof}
Let $\rt*\mc{P}\subseteq\rt*\mc{Q}$. Then $\mc{P}=\mc{U}*\mc{Q}*\mc{U}$ for some $\mc{U}\in\rt$. Further,  $\mc{P}=\mc{U}*\mc{Q}=\mc{U}*\mc{Q}^2=\mc{P}*\mc{Q}$. \\ 
 From the condition $\mc{Q}*\rt\subseteq\mc{P}*\rt$, we have 
  $\mc{Q}=\mc{P}*\mc{V}=\mc{P}^2*\mc{V}=\mc{P}*\mc{Q}$ for some $\mc{V}\in\rt$.  Thus $\mc{P}=\mc{Q}$.
\end{proof}
\begin{lemma}\label{lem3.18}
Let $\mc{P},\mc{Q}\in\rt$ be symmetric and idempotent tensors. If  $\rt*\mc{P}=\rt*\mc{Q}$ or $\mc{Q}*\rt=\mc{P}*\rt$, then $\mc{P}=\mc{Q}$.
\end{lemma}

\begin{proof}
If $\rt*\mc{P}=\rt*\mc{Q}$. Then $\mc{P}=\mc{U}*\mc{Q}*\mc{U}$ and $\mc{Q}=\mc{V}*\mc{P}$ for some $\mc{U},\mc{V}\in\rt$. This yields 
\begin{center}
  $\mc{P}=\mc{U}*\mc{Q}^2=\mc{P}*\mc{Q}$ and $\mc{Q}=\mc{V}*\mc{P}^2=\mc{Q}*\mc{P}$.
\end{center}
 Now 
  $\mc{Q}=\mc{Q}^T=(\mc{Q}*\mc{P})^T=\mc{P}^T*\mc{Q}^T=\mc{P}*\mc{Q}=\mc{P}$.  
\end{proof}

\begin{theorem}
Let $\mc{A}\in\rt$. If $\mc{X}=\mc{A}^\dagger$, then there exist symmetric idempotents $\mc{P}$, $\mc{Q}\in\rt$ such that
\begin{enumerate}
    \item[(i)]  $\mc{P}*\rt=\mc{A}*\rt$ and $\rt*\mc{Q}=\rt*\mc{A}$.
    \item[(ii)] $\textup{lann}(\mc{A})=\textup{lann}(\mc{P})$  and $\textup{rann}(\mc{A})=\textup{rann}(\mc{Q})$.
\end{enumerate}
\end{theorem}
\begin{proof}
(i) Let $\mc{X}=\mc{A}^\dagger$ and define $\mc{P}=\mc{A}*\mc{X}$ and $\mc{Q}=\mc{X}*\mc{A}$. It is trivial that $\mc{P}$ and $\mc{Q}$ are symmetric and idempotents. Further $\mc{P}*\rt\subseteq\mc{A}*\rt$ and $\rt*\mc{Q}\subseteq\rt*\mc{A}$.  From $\mc{A}=\mc{A}*\mc{X}*\mc{A}=\mc{P}*\mc{A}=\mc{Q}*\mc{A}$, we obtain $\mc{A}*\rt\subseteq \mc{P}*\rt$ and $\rt*\mc{A}\subseteq \rt*\mc{Q}$. Hence  $\mc{P}*\rt=\mc{A}*\rt$ and $\rt*\mc{Q}=\rt*\mc{A}$.\\
(ii)  $\textup{lann}(\mc{A})=\textup{lann}(\mc{P})$  and $\textup{rann}(\mc{A})=\textup{rann}(\mc{Q})$ follows from Proposition \ref{pro1}  and Proposition \ref{prop2}.
 \end{proof}
 \begin{corollary}
 Let $\mc{A}\in\rt$. If there exist symmetric idempotents $\mc{P}$, $\mc{Q}\in\rt$ such that
 $\textup{lann}(\mc{A})=\textup{lann}(\mc{P})$  and $\textup{rann}(\mc{A})=\textup{rann}(\mc{Q})$ then 
 \begin{enumerate}
     \item[(i)] $\mc{A}^\dagger=\mc{Q}*\mc{A}^{(1)}*\mc{P}$;
     \item[(ii)] $\mc{P}$ and $\mc{Q}$ are unique. 
     \end{enumerate}
  \end{corollary}
 \begin{proof}
 (i) Let  $\textup{lann}(\mc{A})=\textup{lann}(\mc{P})$  and $\textup{rann}(\mc{A})=\textup{rann}(\mc{Q})$. Then by Proposition \ref{pro1}  and Proposition \ref{prop2}, we obtain  $\mc{A}*\rt=\mc{P}*\rt$ and $\rt*\mc{A}=\rt*\mc{Q}$. Using these  conditions, we obtain 
 \begin{center}
     $\mc{A}=\mc{P}*\mc{A}=\mc{A}*\mc{Q},~ \mc{P}=\mc{A}*\mc{A}^{(1)}*\mc{P}$ and $\mc{Q}=\mc{Q}*\mc{A}^{(1)}*\mc{A}$.
 \end{center}
 Let $\mc{Y} =\mc{Q}*\mc{A}^{(1)}*\mc{P}$. Then $\mc{A}*\mc{Y}*\mc{A}=\mc{A}*\mc{Q}*\mc{A}^{(1)}*\mc{P}*\mc{A}=\mc{A}=\mc{A}*\mc{A}^{(1)}*\mc{A}=\mc{A}$,
 \begin{center}
     $\mc{Y}*\mc{A}*\mc{Y}=\mc{Q}*\mc{A}^{(1)}*\mc{P}*\mc{A}*\mc{Q}*\mc{A}^{(1)}*\mc{P}=\mc{Q}*\mc{A}^{(1)}*\mc{P}=\mc{Y}$,\\
    $\mc{A}*\mc{Y}=\mc{A}*\mc{Q}*\mc{A}^{(1)}*\mc{P}=\mc{A}*\mc{A}^{(1)}*\mc{P}=\mc{P}=\mc{P}^T=(\mc{A}*\mc{Y})^T$, and\\  
    $\mc{Y}*\mc{A}=\mc{Q}*\mc{A}^{(1)}*\mc{P}*\mc{A}=\mc{Q}*\mc{A}^{(1)}*\mc{A}=\mc{Q}=\mc{Q}^T=(\mc{Y}*\mc{A})^T$.
 \end{center}
 Hence $\mc{Y}=\mc{Q}*\mc{A}^{(1)}*\mc{P}$ is the Moore-Penrose inverse of $\mc{A}$.\\
 (ii) The uniqueness of $\mc{P}$ and $\mc{Q}$ are follows from  Lemma \ref{lem3.17} and \ref{lem3.18}.
 \end{proof}

The following result follows from the definition of the Moore-Penrose inverse.

\begin{theorem}
Let $\rt$ be an associative ring with $\mc{I}$. Let $\mc{A}\in\rt$. Then the following statements are true.
\begin{enumerate}
    \item[(i)] If $\mc{A}$ is symmetric and idempotent, then $\mc{A}^{\dg} =
\mc{A}$.
\item[(ii)] $\mc{A} * \mc{A}^\dg$, $\mc{A}^\dg *\mc{A}$, $\mc{I}
-\mc{A}*\mc{A}^\dg$ and
 $\mc{I} - \mc{A}^\dg *\mc{A}$ are all idempotent.
 \item[(iii)]  $\mc{A}^{\dg} = \mc{A}^T$ if and only if $\mc{A}*\mc{A}^T * \mc{A}=\mc{A}$.
\end{enumerate}
\end{theorem}

A characterization of the Moore-Penrose inverse is given by the following: 

\begin{theorem}
Let $\mc{A}, \mc{Y} \in \rt$. The following  are equivalent:   
\begin{enumerate}
 \setlength{\parskip}{0pt}
    \item[(i)] $\mc{A}$ is the Moore-Penrose invertiable and $\mc{Y} = \mc{A}^\dg$.
    \item[(ii)] $\mc{A}*\mc{Y}*\mc{A} = \mc{A}$, $\mc{A}*\rt = \mc{A}^T*\rt$ and $\rt* \mc{Y} = \rt *\mc{A}^T$.
    \item[(ii)]  $\mc{A}*\mc{Y}*\mc{A} = \mc{A}$, $\textup{lann}(\mc{Y}) = \textup{lann}(\mc{A}^T)$ and $\textup{rann}(\mc{Y}) = \textup{rann}(\mc{A}^T)$.
        \item[(iv)]  $\mc{A}*\mc{Y}*\mc{A} = \mc{A}$, $\mc{Y}*\rt = \mc{A}^T*\rt$ and $\rt *\mc{Y} = \rt* \mc{A}^T$.
    \item[(v)]  $\mc{A}*\mc{Y}*\mc{A} = \mc{A}$, $\textup{lann}(\mc{A}^T) \subseteq \textup{lann}(\mc{Y})$ and $\textup{rann}(\mc{A}^T) \subseteq \textup{rann}(\mc{Y})$.
\end{enumerate}
\end{theorem}
\begin{proof}
$(i) \Rightarrow (ii)$. Using the definition of the Moore-Penrose inverse of a tensor we can write 
\begin{equation*}
   \mc{A}^T = \mc{Y}*\mc{A}*\mc{A}^T=\mc{A}^T*\mc{A}*\mc{Y} \textnormal{~~~and~~~} \mc{Y} = \mc{A}^T*\mc{Y}^T*\mc{Y}= \mc{Y}*\mc{Y}^T*\mc{A}.
\end{equation*}
Thus $ \mc{Y}*\rt = \mc{A}^T*\rt \textnormal{~~~and~~~}\rt*\mc{Y} = \rt *\mc{A}^T$.\\
$(ii) \Rightarrow (iii) \Rightarrow (iv) \Rightarrow (v).$ It follows from  proposition \ref{pro1} and proposition \ref{prop2}.\\
$(v) \Rightarrow (i)$.  From  $\mc{A}^T*\mc{Y}^T*\mc{A}^T= \mc{A}^T$, we obtain
\begin{equation*}
    (\mc{I}-\mc{Y}^T*\mc{A}^T) \in \textup{rann}(\mc{A}^T) \subseteq \textup{rann}(X)\textnormal{~~~and~~~}  (\mc{I}-\mc{A}^T*\mc{Y}^T) \in \textup{lann}(\mc{A}^T) \subseteq \textup{lann}(\mc{Y}).
\end{equation*}
Thus $\mc{Y}= \mc{Y}*\mc{Y}^T*\mc{A}^T$ and $\mc{Y}= \mc{A}*\mc{Y}^T*\mc{Y}.$ This yields $\mc{A}*\mc{Y} = \mc{A}*\mc{Y}*(\mc{A}*\mc{Y})^T$ and  $\mc{Y}*\mc{A} = (\mc{Y}*\mc{A})^T*\mc{Y}*\mc{A}$. Therefore, $\mc{A}*\mc{Y}$ and $\mc{Y}*\mc{A}$ are symmetric. Further, $\mc{Y}*\mc{A}*\mc{Y} = \mc{Y}*(\mc{A}*\mc{Y})^T = \mc{Y}*\mc{Y}^T*\mc{A}^T=\mc{Y}$. Hence $\mc{A}^\dg=\mc{Y}$.
This completes the proof.
\end{proof}

\begin{proposition}
Let  $\mc{A} \in \rt$
 and $\mc{A}*\mc{A}^\dg = \mc{A}^\dg *\mc{A}$. Then 
\begin{enumerate}
 \setlength{\parskip}{0pt}
\item[(i)] there exists a
$\mc{X} \in\rt$ such that $\mc{A}*\mc{X} = \mc{A}^T$;
\item[(ii)] there exists a
 $\mc{Y} \in \rt$
 such that $\mc{A}^T *\mc{Y} = \mc{A}$.
\end{enumerate}
\end{proposition}

\begin{proof}
 Let $\mc{A}*\mc{A}^\dg = \mc{A}^\dg *\mc{A}$. Then
  $\mc{A}^T=(\mc{A}*\mc{A}^\dagger*\mc{A})^T=(\mc{A}*\mc{A}*\mc{A}^\dagger)^T=\mc{A}*\mc{A}^\dagger*\mc{A}^T=\mc{A}*\mc{X}$,   where $\mc{X}=\mc{A}^\dagger*\mc{A}^T\in\rt$. 
  Similarly, the second part follows from 
    $\mc{A}=\mc{A}*\mc{A}^\dagger*\mc{A}=\mc{A}^\dagger*\mc{A}*\mc{A}=\mc{A}^T*(\mc{A}^\dagger)^T*\mc{A}=\mc{A}^T*\mc{Y}$,  where $\mc{Y}=(\mc{A}^\dagger)^T*\mc{A}$.
\end{proof}

It is worth mentioning that Liang and Zheng explored in \cite{liang2019} some identities for Moore-Penrose inverse of a tensor. Our next result discusses identities over a ring.

\begin{theorem}\label{th3.31}
Let $\mc{A}\in \rt$. Then the following statements are true.
\begin{enumerate}
 \setlength{\parskip}{0pt}
    \item[(i)] $(\mc{A}^T * \mc{A})^{\dg}=\mc{A}^{\dg}*({\mc{A}^T})^{\dg}$ and $(\mc{A} * \mc{A}^T)^{\dg}=(\mc{A}^T)^{\dg}*\mc{A}^{\dg}.$
    \item[(ii)] $\mc{A}^{\dg} = ({\mc{A}^T}*{\mc{A})^{\dg}}*\mc{A}^T = \mc{A}^T*(\mc{A}*{\mc{A}^T)^{\dg}}$.
\end{enumerate}
  \end{theorem}

In the case of the Moore-Penrose inverse of tensors over a ring with involution, the reverse order law, i.e., $(\mc{A}*\mc{B})^\dg = \mc{B}^\dg*\mc{A}^\dg$, is not true in general. This can be seen from the Example \ref{Ex23} mentioned after our remark.
\begin{remark}
 Theorem \ref{th3.31} (i) is not true if we replace
$\mc{A}^T$ by any other tensor $\mc{B}$, i.e., $(\mc{A}*
\mc{B})^\dg \neq \mc{B}^\dg * \mc{A}^\dg$, where $\mc{A}$  and $
\mc{B} \in \mathbb{R}^{n\times n \times n_3\times n_4 \times \cdots\times n_p}$.
\end{remark}
\begin{example}\label{Ex23}
Let $ \mathcal{A}, \mathcal{B} \in \mathbb{R}^{2\times 2 \times 2}$  with frontal slices 
\begin{eqnarray*}
\mc{A}_{(1)} =
    \begin{pmatrix}
    1 & 1\\
    1 & 0 \\
    \end{pmatrix},~
\mc{A}_{(2)}=
    \begin{pmatrix}
     0 & 1 \\
      1 & 0
    \end{pmatrix},~~
\mc{B}_{(1)} =
    \begin{pmatrix}
    1 & 0 \\
     0 & 1
    \end{pmatrix},~
\mc{B}_{(2)} =
    \begin{pmatrix}
     0 & 2\\
     0 & 0
    \end{pmatrix}.
    \end{eqnarray*}
 Now 
\begin{eqnarray*}
\mc{A}*\mc{B}&=&
\text{fold}( \text{circ}( \text{unfold}(\mc{A})) *  \text{unfold}(\mc{B}))\\
&=&\text{fold}\left(
    \begin{bmatrix}
     \mc{A}_1 & \mc{A}_2\\
     \mc{A}_2 & \mc{A}_1
\end{bmatrix} * 
\begin{bmatrix}
     \mc{B}_1 \\
     \mc{B}_2
\end{bmatrix}
\right) = \text{fold}\left( 
\begin{bmatrix}
     \mc{A}_1*\mc{B}_1+\mc{A}_2*\mc{B}_2 \\
     \mc{A}_2*\mc{B}_1+\mc{A}_1*\mc{B}_2
\end{bmatrix}
\right)
= \text{fold}\left( 
\begin{bmatrix}
     1 & 1\\
     1 & 2\\\hline
     0 & 3\\ 
     1 & 2
\end{bmatrix}
\right),
\end{eqnarray*}
\begin{eqnarray*}
\mc{A}^\dg=
\text{fold}\left( 
\begin{bmatrix}
     1/2 & 1/4\\
     1/4 & -1/8\\\hline
     -1/2 & 1/4\\ 
     1/4 & -1/8
\end{bmatrix}
\right)~~\textnormal{and~~}~~\mc{B}^\dg
= \text{fold}\left( 
\begin{bmatrix}
     1 & 0\\
     0 & 1\\\hline
     0 & -2\\ 
     0 & 0
\end{bmatrix}
\right).
\end{eqnarray*}
We can verify that
\begin{eqnarray*}
\text{fold}\left( 
\begin{bmatrix}
     -2/5 & 1/2\\
     1/20 & -1/8\\\hline
     -3/5 & 1/2\\ 
     9/20 & -1/8
\end{bmatrix}
\right) = (\mc{A}*\mc{B})^\dg \neq \mc{B}^\dg*\mc{A}^\dg
= \text{fold}\left( 
\begin{bmatrix}
     0 & 1/2\\
     1/4 & -1/8\\\hline
     -1 & 1/2\\ 
     1/4 & -1/8
\end{bmatrix}
\right).
\end{eqnarray*}
\end{example}

Our next result deals with the commutative property of $\mc{A}$ and $\mc{A}^\dagger$.

\begin{theorem}
Let $\mc{A}\in \rt$. If $\mc{A} * \mc{A}^T = \mc{A}^T * \mc{A}$, then  $\mc{A} * \mc{A}^{\dg} = \mc{A}^{\dg} * \mc{A}$.
\end{theorem}

\begin{proof}
By using the definition of the Moore-Penrose inverse, we obtain $
\mc{A}*\mc{A}^{\dg} = (\mc{A}^{\dg})^T *\mc{A}^{\dg} * \mc{A} *\mc{A}^T.$ Using $ \mc{A}* \mc{A}^T = \mc{A}^T * \mc{A} $, and Theorem \ref{th3.31}, we get  
\begin{eqnarray*}
\mc{A}*\mc{A}^{\dg} &=&(\mc{A}^{T})^\dg *\mc{A}^{\dg} * \mc{A} *\mc{A}^T= (\mc{A}*\mc{A}^T)^{\dg}*\mc{A}^T*\mc{A}
=(\mc{A}^T*\mc{A})^{\dg}*\mc{A}^T*\mc{A}=\mc{A}^\dg*A.
\end{eqnarray*}
\end{proof}
The converse of the above result is not true in general as shown by the next example.
\begin{example}
Let $ \mathcal{A} \in \mathbb{R}^{2\times 2 \times 3}$, where  
\begin{eqnarray*}
\mc{A}_{(1)} =
    \begin{pmatrix}
    1 & 2\\
    0 & 3 \\
    \end{pmatrix},~
\mc{A}_{(2)} =
    \begin{pmatrix}
     1 & 4 \\
      -2 & -1
    \end{pmatrix},~\textnormal{and}~
\mc{A}_{(3)}=
    \begin{pmatrix}
     0 & 2 \\
      1 & 3
    \end{pmatrix}.
    \end{eqnarray*}
Let $\mc{B}$ be the transpose of $\mc{A}$. Then 
\begin{eqnarray*}
\mc{B}_{(1)} =
    \begin{pmatrix}
    1 & 0\\
    2 & 3 \\
    \end{pmatrix},~
\mc{B}_{(2)} =
    \begin{pmatrix}
     0 & 1 \\
      2 & 3
    \end{pmatrix},~
\mc{B}_{(3)}=
    \begin{pmatrix}
     1 & -2 \\
      4 & -1
    \end{pmatrix}.
    \end{eqnarray*}
We can verify that  
$ \mc{A}^\dg*\mc{A}= \mc{A}*\mc{A}^\dg$, 
\begin{eqnarray*}
\text{fold}\left( 
\begin{bmatrix}
     7 & 11\\
     11 & 43\\\hline
     -1 & 3\\ 
    -3 & 23\\\hline
     -1 & -3\\ 
     3 & 23
\end{bmatrix}
\right) = \mc{A}^T*\mc{A} \neq \mc{A}*\mc{A}^T
= \text{fold}\left( 
\begin{bmatrix}
     26 & 6\\
     6 & 24\\\hline
     21 & 17\\ 
    15 & 1\\\hline
     21 & 15\\ 
     17 & 1
\end{bmatrix}
\right).
\end{eqnarray*}
\end{example}

We next discuss how to compute the Moore-Penrose inverse through lower triangular tensors. 

\begin{theorem}
Let $\mc{A}\in \rt$. Suppose there exist a lower triangular tensor $\mc{L} \in \rt$ such that $\mc{A}*\mc{A}^T=\mc{L}*\mc{L}^T$ and $\mc{L}^L*\mc{L}$ is invertible, then  
\begin{equation}
\mc{A}^\dg = \mc{A}^T*\mc{L}*(\mc{L}^T*\mc{L})^{-2}*\mc{L}^T. 
\end{equation}
\end{theorem}
\begin{proof}
By Theorem \ref{th3.31}, it is enough that show only $(\mc{L}*\mc{L}^T)^\dagger=\mc{L}*(\mc{L}^T*\mc{L})^{-2}*\mc{L}^T$. Let $\mc{X}=\mc{L}*(\mc{L}^T*\mc{L})^{-2}*\mc{L}^T$ and $\mc{A}=\mc{L}*\mc{L}^T$. Then
\begin{enumerate}
 \setlength{\parskip}{0pt}
\item[$\bullet$] 
    $\mc{A}*\mc{X}*\mc{A}=\mc{L}*\mc{L}^T*\mc{L}*(\mc{L}^T*\mc{L})^{-2}*\mc{L}^T*\mc{L}*\mc{L}^T=\mc{L}*\mc{L}^T=\mc{A}$,
\item[$\bullet$] 
    $\mc{X}*\mc{A}*\mc{X}=\mc{L}*(\mc{L}^T*\mc{L})^{-2}*\mc{L}^T*\mc{L}*\mc{L}^T*\mc{L}*(\mc{L}^T*\mc{L})^{-2}*\mc{L}^T=\mc{L}*(\mc{L}^T*\mc{L})^{-2}*\mc{L}^T=\mc{X}$,
\item[$\bullet$]
    $(\mc{A}*\mc{X})^T=(\mc{L}*(\mc{L}^T*\mc{L})^{-1}*\mc{L}^T)^T=\mc{L}*(\mc{L}^T*\mc{L})^{-1}*\mc{L}^T=\mc{A}*\mc{X}$, and
\item[$\bullet$]     
    $(\mc{X}*\mc{A})^T=(\mc{L}*(\mc{L}^T*\mc{L})^{-1}*\mc{L}^T)^T=\mc{L}*(\mc{L}^T*\mc{L})^{-1}*\mc{L}^T=\mc{A}*\mc{X}$.
\end{enumerate}
Hence $(\mc{L}*\mc{L}^T)^\dagger=\mc{L}*(\mc{L}^T*\mc{L})^{-2}*\mc{L}^T$.
\end{proof}
In connection with the above theorem we present Algorithm- \ref{ALgoMPIalgo} for computing the  Moore-Penrose inverse of a tensor $\mc{A} \in \rt$. It is worth noting that the Matlab functions (i.e., zeros, sqrt) are used in Algorithm-\ref{ALgoMPIalgo} on line-8 and line-13 to compute ``zeros'' and ``square root'' of matrices in the Fourier domain, respectively. Table-1 demonstrated the efficiency of the proposed Algorithm \ref{ALgoMPIalgo} in terms of time for computing the Moore-Penrose inverse by comparing the different order of random symmetric tensor with the Algorithm-3 in \cite{liangBing2019}.

\begin{algorithm}[hbt!]
\caption{Computation of Moore-Penrose inverse of a tensor $\mc{A}$} \label{ALgoMPIalgo}
\begin{algorithmic}[1]
\Procedure{MPI}{$\mc{A}$}
\State \textbf{Input} $p$, $n,$ $n_3,\ldots,n_p$ and the tensor $\mathcal{A} \in \rt$.
\For{$i \gets 1$ to $p$} 
\State $\mc{A}=\textup{fft}(\mc{A}, [~ ], i);$
\EndFor
\State $C=n_3n_4\cdots n_p$.
\For{$i \gets 1$ to $C$}
\State $r=0;$~~$\mc{L}=\textup{zeros}(size(\mc{A}(:,:,i)));$
\For{$K \gets 1$ to $n$}
\State $r=r+1;$
\State $\mc{L}(k:n,r,i)=A(k:n,k,i)-\mc{L}(k:n,1:(r-1),i)*\text{transpose}(\mc{L}(k,1:(r-1),i));$
\If ~ $L(k,r,i)>\epsilon$ 
\State $\mc{L}(k,r,i)=\textup{sqrt}(\mc{L}(k,r,i));$
\If ~ $k<n$
\State $\mc{L}((k+1):n,r,i)=\mc{L}((k+1):n,r,i)/\mc{L}(k,r,i);$
\EndIf
\Else
\State $r=r-1;$
\EndIf
\EndFor 
\State $\mc{L}(:,:,i)=\mc{L}(:,1:r, i);$
\EndFor 
 \State Compute $\mc{W}=\mc{A}^T*\mc{L}*(\mc{L}^T*\mc{L})^{-2}*\mc{L}^T$.
      \For{$i \gets p$ to $1$} 
      \State $\mc{X} \leftarrow \textup{ifft}(\mc{W}, [~ ], i);$
      \EndFor
      \State \textbf{return} $\mc{X}$
    \EndProcedure
  \end{algorithmic}
\end{algorithm}

\begin{table}[hbt!]
    \centering
     \caption{Comparison analysis for computing Moore-Penrose inverse of random symmetric tensor $\mc{A}$}\label{TableMoore}
             \begin{tabular}{cccc}
    \hline
     Order of $\mc{A}$ & MT & Algorithms  \\
       \hline
        \multirow{2}{*} {$200\times 300\times 400$} & 
       0.034663 &  In \cite{liangBing2019},  Algorithm-3 \\
     &  0.022151 &   Algorithm-\ref{ALgoMPIalgo}  \\
     \hline
       \multirow{2}{*} {$300\times 400\times 500$} &    
       0.061981 &  In \cite{liangBing2019}, Algorithm-3  \\
     & 0.037659  &   Algorithm-\ref{ALgoMPIalgo}  \\
    \hline
      \multirow{2}{*} {$400\times 500\times 600$} & 
      0.469042 & In \cite{liangBing2019}, Algorithm-3 \\
     & 0.135589 &   Algorithm-\ref{ALgoMPIalgo}  \\
     \hline
    \multirow{2}{*} {$500\times 600\times 700$} &  
    0.565222 &  In \cite{liangBing2019}, Algorithm-3 \\
 &  0.149840  &   Algorithm-\ref{ALgoMPIalgo}  \\
     \hline
     \end{tabular}
        \label{tab:table2}
\end{table}

\subsection{Weighted Moore-Penrose inverse}
We introduce generalized weighted Moore-Penrose inverse  an element over a ring $\rt$ as follows:

\begin{definition}\label{43}
Let $\mc{A},\mc{M}, \mc{N} \in \rt$, where $\mc{M},\mc{N}$ are invertible hermitian tensors. If a  tensor $\mc{Y}\in A\{1,2\}$ satisfies 
\begin{eqnarray*}
(3)~(\mc{M}* \mc{A}*\mc{Y})^T = \mc{M}*\mc{A}*\mc{Y};~~~~(4)~(\mc{N}* \mc{Y}*\mc{A})^T =\mc{N}* \mc{Y}*\mc{A},
\end{eqnarray*}
then $\mc{Y}$ is called the generalized weighted Moore-Penrose inverse of $\mc{A}$ and denoted by $\mc{A}_{\mc{M},\mc{N}}^\dg$.
\end{definition}

The uniqueness of the generalized weighted Moore-Penrose inverse is proved in the next result.
\begin{proposition}
Let $\mc{A} \in \rt$, and a pair of invertible hermitian tensors $\mc{M} \in \rt$ and $\mc{N} \in \rt$ be given. If the generalized weighted Moore-Penrose inverse exists then it is unique.
\end{proposition}
\begin{proof}
Suppose there exist $\mc{X}_1$ and $\mc{X}_2$ both satisfying the condition $(1)-(4)$. Then 
\begin{eqnarray*}
\mc{X}_1&=&\mc{X}_1*\mc{A}*\mc{X}_1=\mc{N}^{-1}*\mc{A}^T*{\mc{X}_1}^T*\mc{N}*\mc{X}_1
= \mc{N}^{-1}*\mc{A}^T*{\mc{X}_2}^T*\mc{A}^T*{\mc{X}_1}^T*\mc{N}*\mc{X}_1\\
&=& \mc{N}^{-1}*\mc{A}^T*{\mc{X}_2}^T*\mc{N}*\mc{X}_1*\mc{A}*\mc{X}_1= \mc{N}^{-1}*\mc{A}^T*{\mc{X}_2}^T*\mc{N}*\mc{X}_1=\mc{X}_2*\mc{A}*\mc{X}_1,\\
\mc{X}_2 &=& \mc{X}_2*\mc{A}*\mc{X}_2 = \mc{X}_2*\mc{M}^{-1}*{\mc{X}_2}^T*\mc{A}^T*\mc{M}
= \mc{X}_2*\mc{M}^{-1}*{\mc{X}_2}^T*\mc{A}^T*{\mc{X}_1}^T*\mc{A}^T*\mc{M}\\
&=& \mc{X}_2*\mc{M}^{-1}*{\mc{X}_2}^T*\mc{A}^T*\mc{M}*\mc{A}*\mc{X}_1=\mc{X}_2*\mc{A}*\mc{X}_2*\mc{A}*\mc{X}_1 = \mc{X}_2*\mc{A}*\mc{X}_1.
\end{eqnarray*}
Hence $\mc{X}_1=\mc{X}_2*\mc{A}*\mc{X}_1=\mc{X}_2$.
\end{proof}

The existence and computation of the generalized weighted Moore-Penrose inverse is discussed below.
\begin{theorem}\label{2.11}
Let $\mc{A}\in \rt$ and $\mc{M},~\mc{N} \in \rt$ be invetible hermitian tensors. Then the following statements are equivalent:
\begin{enumerate}
 \setlength{\parskip}{0pt}
\item [(i)] $\mc{A}^\dagger_{\mc{M},\mc{N}}$ exists.
\item [(ii)] There exist unique idempotent tensors $\mc{P},~\mc{Q} \in \rt$ such that
\begin{center}
 $\mc{M}*\mc{P}=(\mc{M}*\mc{P})^T$, $\mc{N}*\mc{Q}=(\mc{N}*\mc{Q})^T$, $\mc{P}*\rt=\mc{A}*\rt$,  and $\rt*\mc{Q}=\rt*\mc{A}$.   
\end{center}
\end{enumerate} 
If any one of the statements $(i),~(ii)$ holds, then  $\mc{A}^\dagger_{M,N}=\mc{Q}*\mc{A}^{(1)}*\mc{P}$ and thus is invariant for any choice of $\mc{A}^{(1)}$. 
\end{theorem}

\begin{proof}
$(i)\Rightarrow (ii)$.  Let $\mc{X}=A^\dagger_{M,N}$. If we define $\mc{P}=\mc{A}*\mc{X}$ and $\mc{Q}=\mc{X}*\mc{A}$, then  $\mc{P}*\rt=\mc{A}*\rt$ and $\rt*\mc{Q}=\rt*\mc{A}$ can be shown easily. Further $\mc{P}=\mc{A}*\mc{X}=\mc{A}*\mc{X}*\mc{A}*\mc{X}=\mc{P}^2$, $\mc{Q}=\mc{X}*\mc{A}=\mc{X}*\mc{A}*\mc{X}*\mc{A}=\mc{Q}^2$, $\mc{M}*\mc{P}=\mc{M}*\mc{A}*\mc{X}=(\mc{M}*\mc{A}*\mc{X})^T=(\mc{M}*\mc{P})^T$, and $\mc{N}*\mc{Q}=\mc{N}*\mc{X}*\mc{A}=(\mc{N}*\mc{X}*\mc{A})^T=(\mc{N}*\mc{Q})^T$. To show the uniqueness of $\mc{P}$ and $\mc{Q}$, suppose there exists two idempotent pairs $(\mc{P},\mc{Q})$ and $(P_1,\mc{Q}_1)$ which satisfies $(b)$. Now from $\rt*\mc{Q}=\rt*\mc{Q}_1$, we have $\mc{Q}=\mc{Q}*\mc{Q}_1$ and $\mc{Q}_1=\mc{Q}_1*\mc{Q}$. From $\mc{Q}=\mc{Q}*\mc{Q}_1$ and $\mc{Q}_1=\mc{Q}_1*\mc{Q}$, we get 
\begin{eqnarray}\label{eq23}
\mc{N}*\mc{Q}=(\mc{N}*\mc{Q})^*=(\mc{N}*\mc{Q}*\mc{N}^{-1}*\mc{N}*\mc{Q}_1)^T=\mc{N}*\mc{Q}_1*\mc{N}^{-1}*\mc{N}*\mc{Q}=\mc{N}*\mc{Q}_1*\mc{Q}=\mc{N}*\mc{Q}_1
\end{eqnarray}
 Pre-multiplying \eqref{eq23} by $N^{-1}$, we obtain $\mc{Q}=\mc{Q}_1$. Using the similar lines, we can prove the uniqueness of $\mc{P}$.\\
$(ii)\Rightarrow (i)$. Let $\mc{P}$ and $\mc{Q}$ be the unique idempotent tensors such that $\mc{M}*\mc{P}=(\mc{M}*\mc{P})^T$, $\mc{N}*\mc{Q}=(\mc{N}*\mc{Q})^T$, $\mc{P}*\rt=\mc{A}*\rt$ and $\rt*\mc{Q}=\rt*\mc{A}$. Then $\mc{A}=\mc{P}*\mc{A}=\mc{A}*\mc{Q}$, $\mc{P}=\mc{A}*\mc{U}$, and $\mc{Q}=\mc{V}*\mc{A}$ for some $\mc{U}, ~\mc{V}\in\rt$. In addition, $\mc{P}=\mc{A}*\mc{A}^{(1)}*\mc{A}*\mc{U}=\mc{A}*\mc{A}^{(1)}*\mc{P}$ and $\mc{Q}=\mc{Q}*\mc{A}^{(1)}*\mc{A}$. Now, consider $\mc{Y}=\mc{Q}*\mc{A}^{(1)}*\mc{P}$. Then $A^\dagger_{M,N}=\mc{Y}=\mc{Q}*\mc{A}^{(1)}*\mc{P}$ is follows from the following identities:
\begin{itemize}
\item $\mc{A}*\mc{Y}*\mc{A}=\mc{A}*\mc{Q}*\mc{A}^{(1)}*\mc{P}*\mc{A}=\mc{A},$
\item $\mc{Y}*\mc{A}*\mc{Y}=\mc{Q}*\mc{A}^{(1)}*\mc{P}*\mc{A}*\mc{Q}*\mc{A}^{(1)}*\mc{P}=\mc{Y}$;
\item $\mc{M}*\mc{A}*\mc{Y}=\mc{M}*\mc{A}*\mc{Q}*\mc{A}^{(1)}*\mc{P}=\mc{M}*\mc{P}=(\mc{M}*\mc{A}*\mc{Y})^T$,
\item $\mc{N}*\mc{Y}*\mc{A}=\mc{N}*\mc{Q}*\mc{A}^{(1)}*\mc{P}*\mc{A}=\mc{N}*\mc{Q}=(\mc{N}*\mc{Y}*\mc{A})^T$.
\end{itemize}

Let $\mc{X}_1$ and $\mc{X}_2$ be two elements of $\mc{A}\{1\}$. From $\mc{P}*\rt=\mc{A}*\rt$ and $\rt*\mc{Q}=\rt*\mc{A}$, we obtain 
\begin{equation}\label{eqq7}
    \mc{P}=\mc{A}*\mc{V}\mbox{ and }\mc{Q}=\mc{U}*\mc{A}~\mbox{ for some }\mc{U},\mc{V}\in\rt.
\end{equation}
Using \eqref{eqq7}, we have 
\begin{eqnarray*}
A^\dagger_{M,N}&=&\mc{Q}*\mc{X}_1*\mc{P}=\mc{U}*\mc{A}*\mc{X}_1*\mc{A}*\mc{V}=\mc{U}*\mc{A}*\mc{V}=\mc{U}*\mc{A}*\mc{X}_2*\mc{A}*\mc{V}=\mc{Q}*\mc{X}_2*\mc{P}.
\end{eqnarray*}
Thus $A^\dagger_{M,N}$ is an invariant for any choice of $\mc{A}^{(1)}$.
\end{proof}

Now we present an algorithm (see the Algorithm-\ref{AlgoSqrt}) for computing square root of a symmetric positive definite  tensor, which will be used for computation of the weighted Moore-Penrose inverse by using the Moore-Penrose inverse. Here, the matrix computation Matlab functions (i.e., eig, sqrt) are utilized in Algorithm-\ref{AlgoSqrt} to compute  square root of a symmetric positive definite tensor $\mc{A} \in \rt$. In fact, the function ``eig'' uses to compute the eigenvalue of matrices in the Fourier domain.
\begin{algorithm}[hbt!]
  \caption{Computation of square root of a symmetric positive definite tensor $\mc{A}$} \label{AlgoSqrt}
  \begin{algorithmic}[1]
    \Procedure{SQRT}{$\mc{A}$}
    \State \textbf{Input} $p$, $n,$ $n_3,\ldots,n_p$ and the tensor $\mathcal{A}\in \rt$.
      \For{$i \gets 1$ to $p$} 
      \State $\mc{A}=\textup{fft}(\mc{A}, [~ ], i);$
      \EndFor
      \State $C=n_3n_4\cdots n_p$
      \For{$i \gets 1$ to $C$} 
       \State $[\mc{V}(:, :, i ), ~\mc{D}(:, :, i )]= \textup{eig}(\mc{A}(:, :, i ));$,
      \EndFor
      \For{$i \gets 1$ to $C$} 
      \State $\mc{S}(:, :, i )= V(:,:,i)\textup{sqrt}(\mc{D}(:,:,i)\textup{inv}(\mc{V}(:,:,i);$ 
      \EndFor
      \For{$i \gets p$ to $1$} 
      \State $\mc{X} \leftarrow \textup{ifft}(\mc{S}, [~ ], i);$
      \EndFor
      \State \textbf{return} $\mc{S}$ \Comment{$\mc{S}$ is the equal to $\mc{A}^{1/2}$}
    \EndProcedure
  \end{algorithmic}
\end{algorithm}

An equivalent characterization for existence of the generalized weighted Moore-Penrose inverse is presented in the next result.

\begin{theorem}\label{GWMPI}
Let $\mc{M},\mc{N}\in \rt$ be an invertible hermitian  tensors and $\mc{A} \in \rt$ .  If $\mc{M}^{1/2}$ and $\mc{N}^{-1/2}$ (the square root of $\mc{M}$ and $\mc{N}^{-1}$ respectively) are exists, then the generalized weighted Moore-Penrose inverse of $\mc{A}$ exists. Moreover 
\begin{center}
    $\mc{A}_{\mc{M},\mc{N}}^\dg = \mc{N}^{-1/2} * ({\mc{M}^{1/2} * \mc{A} *\mc{N}^{-1/2}})^\dg * \mc{M}^{1/2}$.
\end{center}
\end{theorem}

\begin{proof}
Let $\mc{Y}=(\mc{M}^{1/2} * \mc{A} *\mc{N}^{-1/2})^\dg$ and $\mc{X}=\mc{N}^{-1/2}*\mc{Y}*\mc{M}^{1/2}$. From the conditions 
$\mc{M}^{1/2} * \mc{A} *\mc{N}^{-1/2}*\mc{Y}*\mc{M}^{1/2} * \mc{A} *\mc{N}^{-1/2}=\mc{M}^{1/2} * \mc{A} *\mc{N}^{-1/2}$ and $\mc{Y}*\mc{M}^{1/2} * \mc{A} *\mc{N}^{-1/2}*\mc{Y}=\mc{Y}$, we obtain
$\mc{A}*\mc{X}*\mc{A}=\mc{A}$ and $\mc{X}*\mc{A}*\mc{X}=\mc{X}$. Further 
\begin{eqnarray*}
(\mc{M}*\mc{A}*\mc{X})^T&=&(\mc{M}*\mc{A}*\mc{N}^{-1/2}*\mc{Y}*\mc{M}^{1/2})^T=(\mc{M}^{1/2}*\mc{M}^{1/2}*\mc{A}*\mc{N}^{-1/2}*\mc{Y}*\mc{M}^{1/2})^T\\
&=&\mc{M}*\mc{A}*\mc{N}^{-1/2}*\mc{Y}*\mc{M}^{1/2}=\mc{M}*\mc{A}*\mc{X},
\end{eqnarray*}
\begin{equation*}
\begin{split}
(\mc{N}*\mc{X}*\mc{A})^T=&(\mc{N}^{1/2}*\mc{Y}*\mc{M}^{1/2}*\mc{A})^T=(\mc{N}^{1/2}*\mc{Y}*\mc{M}^{1/2}*\mc{A}*\mc{N}^{-1/2}*\mc{N}^{1/2})^T\\
=&\mc{N}^{1/2}*\mc{Y}*\mc{M}^{1/2}*\mc{A}=\mc{N}*\mc{X}*\mc{A}. \qedhere
\end{split}
\end{equation*}
\end{proof}
We employ algorithm-5 for computing the weighted Moore-Penrose inverse. In this algorithm we use the Matlab function ``pinv'' to compute the Moore-Penrose inverse of matrices in the Fourier domain. 
\begin{algorithm}[hbt!]
  \caption{Computation of Weighted Moore-Penrose inverse of a tensor $\mc{A}$} \label{AlgoWMPI}
  \begin{algorithmic}[1]
    \Procedure{WMPI}{$\mc{A}$}
    \State \textbf{Input} $p$, $n,$ $n_3,\ldots,n_p$ and the tensor $\mathcal{A}, \mc{M},\mc{N}\in \rt$.
      \For{$i \gets 1$ to $p$} 
      \State $\mc{A}=\textup{fft}(\mc{A}, [~ ], i);$~~$\mc{M}=\textup{fft}(\mc{M}, [~ ], i);$~~$\mc{N}=\textup{fft}(\mc{N}, [~ ], i);$
      \EndFor
       \State Compute $\mc{S}:=\mc{M}^{1/2}, \mc{T}:=\mc{N}^{-1/2}$ by using Algorithm \ref{AlgoSqrt}.
        \State $C=n_3n_4\cdots n_p$.
      \For{$i \gets 1$ to $C$} 
       \State $\mc{Z}(:, :, i ) = \textup{pinv}((\mc{S}*\mc{A}*\mc{T})(:, :, i ));$
      \EndFor
      \For{$i \gets 1$ to $C$} 
       \State $\mc{W}(:, :, i ) =(\mc{T}*\mc{Z}*\mc{S})(:, :, i ));$
      \EndFor
      \For{$i \gets p$ to $1$} 
      \State $\mc{X} \leftarrow \textup{ifft}(\mc{W}, [~ ], i);$
      \EndFor
      \State \textbf{return} $\mc{X}$
    \EndProcedure
  \end{algorithmic}
\end{algorithm}
Further, the Algorithm \ref{AlgoWMPI} is validated in  the following example. 
\begin{example}
Let $ \mathcal{A}, \mathcal{B} \in \rt $  
 with 
\begin{eqnarray*}
\mc{A}_{(1)} =
    \begin{pmatrix}
    1 & 0 \\
    0 & 1
    \end{pmatrix},~
\mc{A}_{(2)} =
    \begin{pmatrix}
     4 & 2 \\
     -1 & 1
    \end{pmatrix},~~
    \mc{A}_{(3)} =
    \begin{pmatrix}
     1 &  1\\
     1 & 2
    \end{pmatrix},
    \end{eqnarray*}
    
    \begin{eqnarray*}
\mc{M}_{(1)} =
    \begin{pmatrix}
    3 & 1\\
    1 & 6
    \end{pmatrix},~
\mc{M}_{(2)}=
    \begin{pmatrix}
     1 & 1\\
     2 & 5
    \end{pmatrix},~~
    \mc{M}_{(3)} =
    \begin{pmatrix}
    1 & 2\\
    1 & 5
    \end{pmatrix},~~
\mc{N}_{(1)} =
    \begin{pmatrix}
    2 & 0 \\
    0 & 2
    \end{pmatrix},~
\mc{N}_{(2)} =
    \begin{pmatrix}
     1 & 0 \\
     0 & 1
    \end{pmatrix},~~
    \mc{N}_{(3)} =
    \begin{pmatrix}
     1 & 0\\
     0 & 1
    \end{pmatrix}.
    \end{eqnarray*}
Using the Algorithm-\ref{AlgoSqrt} we obtain,
\begin{eqnarray*}
\mc{M}^{1/2}=
fold\left( 
\begin{bmatrix}
      895/557   &      77/1098 \\ 
      77/1098    &   583/305   \\\hline
     131/494      &   77/1098 \\ 
     179/346    &   1372/1349  \\\hline
     131/494     &   179/346\\   
      77/1098    &  1372/1349
\end{bmatrix}
\right)~\textnormal{and~~}\mc{N}^{-1/2}
= fold\left( 
\begin{bmatrix}
      1/6   &         7/24   \\ 
      -1/3   &        -7/12 \\\hline
      -1/6      &     -1/24 \\   
       1/3      &      5/12 \\\hline
       1/6      &     -3/8 \\    
       0        &      5/12  
\end{bmatrix}
\right).
\end{eqnarray*}
By applying Algorithm-\ref{AlgoWMPI}, we get   
\begin{eqnarray*}
\mc{A}_{\mc{M},\mc{N}}^\dg = fold\left( 
\begin{bmatrix}
      3/26      &   -11/26\\    
      -4/13     &      6/13\\\hline 
       9/26     &     -7/26\\    
       1/13     &      5/13 \\\hline
       1/26     &    -21/26\\    
       3/13     &      2/13\\  
\end{bmatrix}
\right).
\end{eqnarray*}
\end{example}

\section{Image deblurring}
Signal and image processing are still a major challenge and has stayed as a  preoccupation for the scientific community. The inclusion of ring theory to the spatial analysis of digital images and computer vision tasks has been carried out in \cite{Lidl94}.  In this section, we apply the Moore-Penrose inverse in image reconstruction problem. The discrete model for the two-dimensional (2D) image blurring (see \cite{hansen}) is represented as
\begin{equation}\label{eq4.11}
   Ax = b,
\end{equation}
where $A$ is the blurring matrix and has some special structure like a banded matrix, Toeplitz or block-Toeplitz matrix (see \cite{calvetti,kilmer11}). Here $x$ is the true image and $b$ is the blurred image. In practice,
 $b$ is corrupted with noise and the blurred matrix $A$ is ill-conditioned. Such type of ill-posed problems are also observed in the discretization of Fredholm integral equations of the first kind,  noisy image restoration, computer tomography, and inverse problems within electromagnetic flow. Ill-posed problems were extensively studied in the context of an inverse problem and image restorations. One can find more details on image restoration and deblurring in \cite{andrews,lagend, MosiC21}.
 
The three-dimensional (3D) color image blurring problem,  often occurs in medical or geographic imaging. It can be written as a tensor equation
\begin{equation}\label{eq4.13}
    \mc{A}*\mc{X}=\mc{B},
\end{equation}
where $\mc{A}$ is the known blurring tensor. Further, $\mc{X}$ and $\mc{B}$ are tensors representing the true image and the blurred image, often corrupt with noise, respectively. In image restoration, the main objective is to establish a blurred free image that requires the approximate solution of a multilinear system given by the equation  \eqref{eq4.13}. To
find the approximate solution of the ill-posed system viz. system \eqref{eq4.13}, several iterative methods such as
preconditioned LSQR (see \cite{kilmer11}), conjugate gradient (CG), $t$-singular value decomposition ($t$-SVD), Golub–Kahan iterative bidiagonalizatios (G-K-Bi-diag), for details see \cite{Kilmer13, guide2020tensor}. 
\begin{figure}[hbt!]\label{im1}
\begin{center}
\subfigure[]{\includegraphics[height=4.77cm, width=5.77cm]{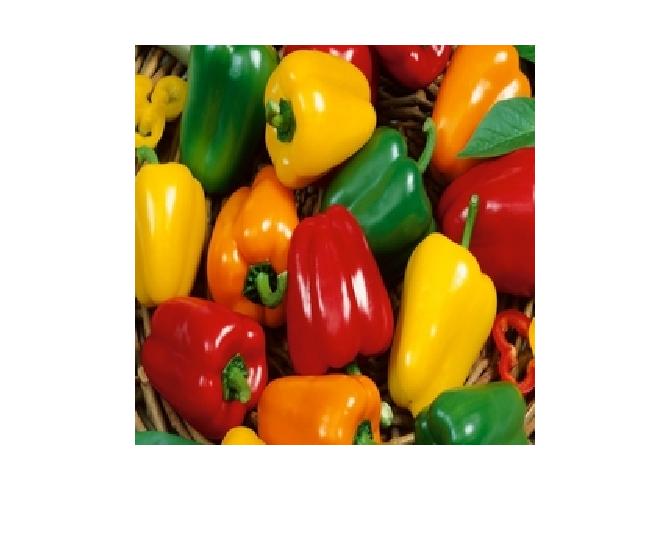}}
\subfigure[]{\includegraphics[height=4.77cm, width=5.77cm]{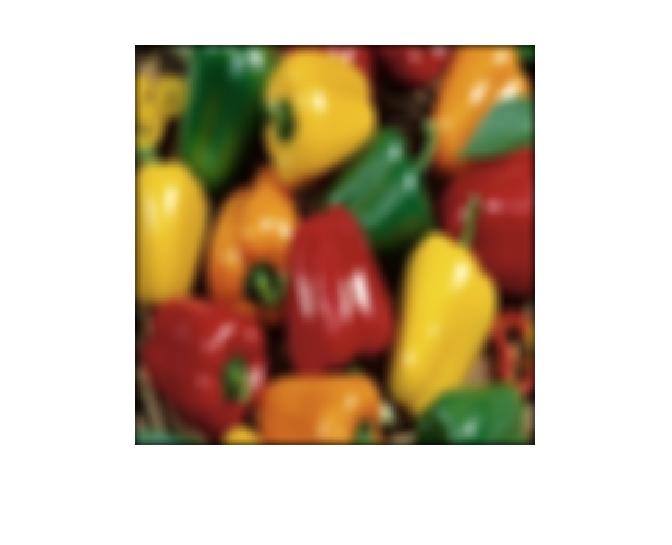}}
\subfigure[]{\includegraphics[height=4.77cm, width=5.77cm]{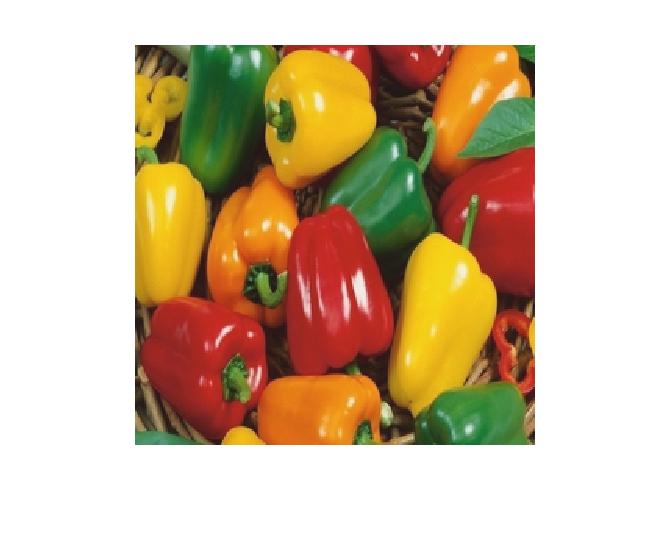}}\\
\subfigure[]{\includegraphics[height=4.77cm, width=5.77cm]{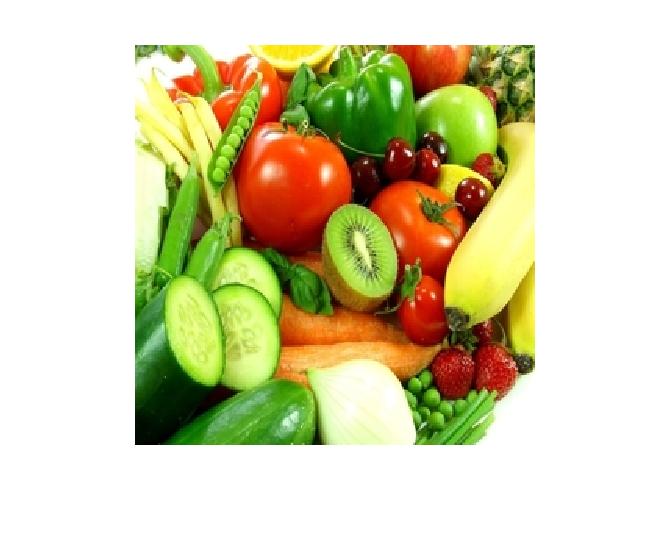}}
\subfigure[]{\includegraphics[height=4.77cm, width=5.77cm]{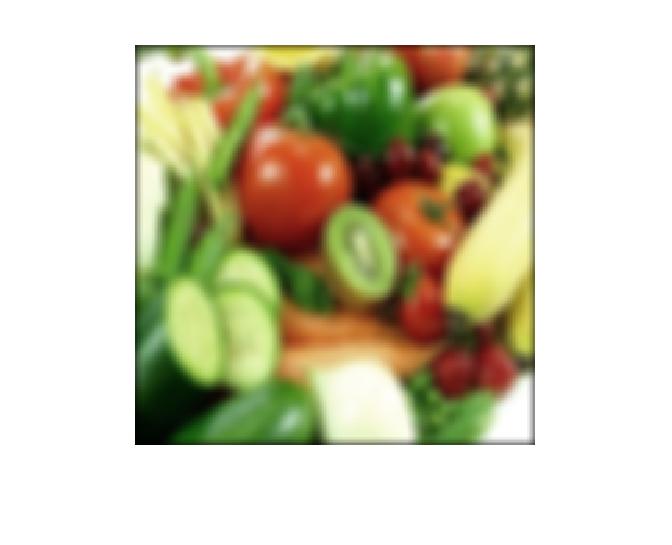}}
\subfigure[]{\includegraphics[height=4.77cm, width=5.77cm]{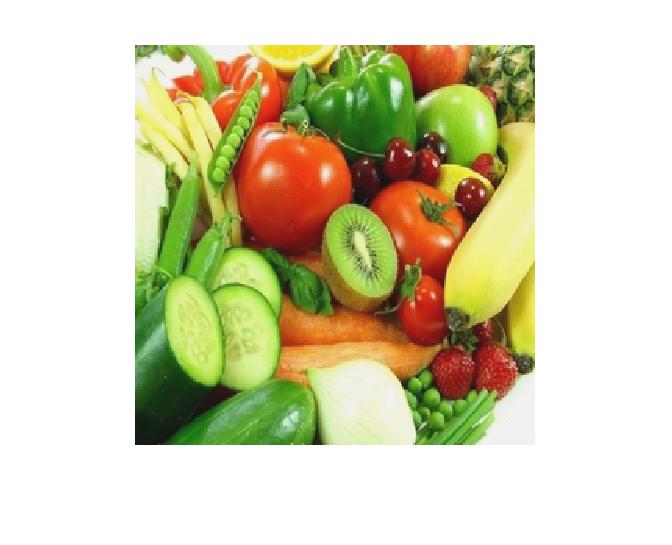}}
\caption{(a) and (d) True image,~~~~~ (b) and (e) Blurred noisy image,~~~~~~~ (c) and (f) Reconstruction image.
}
\end{center}
\end{figure}

It is well known that RGB image is a  third-order tensor, and they often represent the intensities in the red, green, and blue scales. Consider a original error-free color image $\mc{X}$ of size $n \times n \times 3$. Let $X_{(1)}, X_{(2)}$ and $X_{(3)}$ be the slices of size $n\times n$ that constitute the three channels of the image $\mc{X}$ represent the color information. Similarly, consider $B_{(1)}, B_{(2)}$ and $B_{(3)}$ are the slices of size $n\times n$ that associated with error-free blurred color image $\mc{B}$. Consider both cross-channel and within-channel blurring take place in the blurring process of the original image \cite{hansen}. Now, define \textbf{vec} to take an $n \times n$ matrix and return a $n^2 \times 1$ vector by stacking the columns of the matrix from left to right. We now describe the following blurring model as the equivalent form of the Eq.\eqref{eq4.13},
\begin{equation}\label{linmodel}
\left(A^c \otimes {A}^{h}\otimes{A}^{v}\right) \begin{pmatrix}
 \textbf{vec} (X_{(1)})\\
  \textbf{vec} (X_{(2)})\\
   \textbf{vec} (X_{(3)})
\end{pmatrix}=\begin{pmatrix}
 \textbf{vec} (B_{(1)})\\
  \textbf{vec} (B_{(2)})\\
   \textbf{vec} (B_{(3)})
\end{pmatrix},
\end{equation}
where $\otimes$ denotes the Kronecker product of the matrices. ${A}^{h}\in\mathbb{R}^{n\times n}$ and ${A}^{v}\in\mathbb{R}^{n\times n}$ are horizontal and vertical  within-channel blurring matrices, respectively.  Further, $A^c$ is the cross-channel blurring matrix of size $3 \times 3$ as follows: 
\begin{equation*}
A^c=\begin{pmatrix}
c_1 & c_3 & c_2 \\
c_2 & c_1 & c_3 \\
c_3 & c_2 & c_1
\end{pmatrix} \text{~~with~~} \displaystyle\sum_{i=1}^3c_i = 1 \text{~~and~~} c_i \in \mathbb{R} \text{~~for~~} i=1,2,3.
\end{equation*}
Following the circulant structure  of the cross-channel blurring matrix, we can write the following equivalent tensor-tensor model: 
\begin{equation}
\mc{A}* \mc{X} * \mc{C}= \mc{B},
\end{equation}
where $\mc {A}(:,:,k)=c_k {A}^{v}$, for $k=1,2,3$ with $\mc{C}(:,:,1)=({A^{h}})^T$, and $\mc{C}(:,:,2)=\mc{C}(:,:,3)=0$. We construct the blurring tensor $\mc{A}$ using the following symmetric banded Toeplitz matrix:
$$ {A}^{v}_{ij}  = {A}^{h}_{ij} =\left\{\begin{array}{ll}
\frac{1}{\sigma \sqrt{2 \pi}} e^{-\frac{(i-j)^{2}}{2 \sigma^{2}}}, & |i-j| \leq k \\
0, & \text { otherwise }
\end{array}\right.$$
Here $\sigma$ controls the amount of smoothing, i.e., the more ill posed the problem when $\sigma$ is larger.  Further, we use $c_1=0.8$ and $c_2=c_3=0.1$ in the cross-channel blurring matrix $A^c$. 
For numerical experiment, we consider two $256\times256\times 3  $ colour images, and present in Figure 1 (a) and (d). Using $\sigma =4$ and $k=6$, we generate the blurred image $\mc{B} = \mc{A}*\mc{X}*\mc{C} + \mc{N}$. where $\mc{N}$ is a noise tensor distributed normally with mean $0$ and variance $10^{-3}$.  The blurred noisy image is shown in the Figure 1 (b) and (d).
Finally, we have reconstructed the true image using the Moore-Penrose inverse, i.e., the Algorithm \ref{ALgoMPIalgo}. The resulting reconstruction is given in the Figure 1 (c) and (e).

\section{Conclusion}
We have introduced the notion of generalized inverse of tensors over a ring. Our intention is to generalize some known results on generalized inverse of matrices to tensors over the algebraic structure of a ring. Since this tensor product is not a simple extension of the matrix product, we explore effective algorithms for computing generalized inverses, the Moore-Penrose inverse and weighted Moore-Penrose inverse of tensors together with a few supporting algorithms for these inverses, including  transpose of a tensor and square root of a symmetric positive definite tensor. Finally, the algorithm is used to restore the deblurring image via the Moore-Penrose inverse.
\section*{Acknowledgements} 
The first and the third authors are  grateful to the Mohapatra Family Foundation and the College of Graduate Studies of the University of Central Florida for their support for this research.
\section*{Conflict of Interest}
The authors declare no potential conflict of interests.
\section*{ORCID}
Ratikanta Behera\orcidA \href{https://orcid.org/0000-0002-6237-5700}{ \hspace{2mm}\textcolor{lightblue}{ https://orcid.org/0000-0002-6237-5700}}\\
Jajati Keshari Sahoo\orcidB \href{https://orcid.org/0000-0001-6104-5171}{ \hspace{2mm}\textcolor{lightblue}{ https://orcid.org/0000-0001-6104-5171}}\\
R. N. Mohapatra\orcidC \href{https://orcid.org/0000-0002-5502-3934}{ \hspace{2mm}\textcolor{lightblue}{ https://orcid.org/0000-0002-5502-3934}}\\
M. Zuhair Nashed\orcidD \href{https://orcid.org/0000-0002-6009-1720}{ \hspace{2mm}\textcolor{lightblue}{https://orcid.org/0000-0002-6009-1720}}

\end{document}